\documentclass[11pt]{amsart}

\usepackage[T1]{fontenc}
\usepackage[utf8]{inputenc}
\usepackage{lmodern}
\usepackage{amssymb,mathtools}
\usepackage{enumitem}
\usepackage[hidelinks]{hyperref}
\usepackage{microtype}
\usepackage{bm}
\usepackage{array}

\setlist[enumerate]{topsep=4pt,itemsep=3pt,parsep=0pt}

\theoremstyle{plain}
\newtheorem{theorem}{Theorem}[section]
\newtheorem{proposition}[theorem]{Proposition}
\newtheorem*{introproposition}{Proposition}
\newtheorem{lemma}[theorem]{Lemma}
\newtheorem{corollary}[theorem]{Corollary}
\theoremstyle{definition}
\newtheorem{definition}[theorem]{Definition}
\newtheorem{example}[theorem]{Example}
\theoremstyle{remark}
\newtheorem*{remark}{Remark}
\numberwithin{equation}{section}
\DeclareMathOperator{\PSH}{PSH}

\DeclareMathOperator{\dist}{dist}
\DeclareMathOperator{\Rea}{Re}
\newcommand{\C}{\mathbb C}
\newcommand{\R}{\mathbb R}
\newcommand{\Om}{\Omega}
\newcommand{\clOm}{\overline\Omega}
\newcommand{\ph}{\varphi}
\newcommand{\eps}{\varepsilon}
\newcommand{\PB}{\operatorname{PB}}
\newcommand{\Pc}{\operatorname{P}^{c}}
\newcommand{\OO}{\mathcal O}
\newcommand{\ddc}{dd^{c}}
\newcommand{\hull}[2]{\widehat{#1}_{#2}}
\newcommand{\cclass}{\mathcal C_{\ph}(\Om)}
\newcommand{\nclass}{\mathcal N_{\ph}(\Om)}

\allowdisplaybreaks[1]
\hypersetup{pdftitle={Complex Dirichlet problem with unbounded boundary values},pdfauthor={Huy Hoang Dao and Quang Dieu Nguyen}}
\title[Complex Dirichlet problem]{Complex Dirichlet problem with unbounded boundary values}

\author{Huy Hoang Dao}
\address{Hanoi National University of Education, 136 Xuan Thuy, Cau Giay, Hanoi, Vietnam}
\email{daohuyhoang05121999@gmail.com}
\author{Quang Dieu Nguyen}
\address{School of Mathematics and Computer Science, Hanoi National University of Education, 136 Xuan Thuy, Cau Giay, Hanoi, Vietnam}
\email{ngquang.dieu@hnue.edu.vn}

\subjclass[2020]{Primary 32U35; Secondary 32U15, 32E20}
\keywords{Perron--Bremermann envelope, superharmonic obstacle, holomorphic hull, pluripolar set, complex Dirichlet problem}

\begin{document}
	\begin{abstract}
		Let $\Omega\subset\mathbb C^n$ be bounded and let
		$\varphi:\overline\Omega\to(-\infty,+\infty]$ be continuous and
		superharmonic on $\Omega$. We study the maximal plurisubharmonic minorant
		$G_{\varphi,\Omega}$ and Dirichlet problems with unbounded boundary data.
		Our main theorem proves interior continuity from a local biholomorphic
		deformation satisfying a one-sided compatibility condition near the infinite
		locus and barriers on the associated boundary cluster set. The deformed
		domains need not be nested, and the proof uses neither a Runge hypothesis nor
		a pseudoconvex ambient neighbourhood. We also prove nonmonotone
		Perron--Bremermann convergence outside an intrinsic neighbourhood pluripolar
		hull, a disk-fibre hull representation, stability under simultaneous
		nonmonotone variations of the domain and obstacle, and uniqueness in a fixed
		obstacle class.
	\end{abstract}
	\maketitle

	\section{Introduction and statements of results}\label{sec:intro}
The Perron--Bremermann method constructs maximal plurisubharmonic
functions from boundary data. For finite continuous data, the classical work of
Bremermann, Walsh, Bedford--Taylor and Sibony gives a well-developed continuity
and Dirichlet theory; see \cite{Br,Walsh,BT,Si2,Kl}. When the data may take the
value $+\infty$, two difficulties appear immediately: uniform barriers may fail
near the infinite locus, and a boundary trace by itself need not determine a
maximal plurisubharmonic function uniquely.

We study this singular regime through the obstacle envelope
\[
G_{\varphi,\Omega}(z)
=\sup\{v(z):v\in\PSH(\Omega),\ v\le\varphi\text{ on }\Omega\},
\]
where $\varphi:\overline\Omega\to(-\infty,+\infty]$ is continuous and
superharmonic on $\Omega$. For $\varphi=-\log|f|$, adding $\log|f|$ recovers the
pluricomplex Green function with poles along the divisor of $f$, so the setting
contains the Green-function problem considered in \cite{DiGreen}.

The present paper extends several aspects of that 2005 framework. Theorem~2.2
of \cite{DiGreen} treats a finite union of hyperplanes through the origin and
uses a homogeneous defining polynomial, scalar dilations, holomorphic
convexity in a pseudoconvex ambient domain, and good boundary points. Here the
homogeneous model is replaced by a general extended-valued superharmonic
obstacle, and scalar dilation is replaced by a local biholomorphic deformation. The domains
$\Phi_t(\Omega)$ need not be nested: they may expand in some directions and
contract in others. More importantly, continuity is proved independently of
holomorphic-hull arguments. For one small parameter $t$, we
transport the envelope, obtain a one-sided estimate only on the overlap with
$\overline\Omega$, cut off the part of the moving boundary that meets the fixed
domain, regularize on the resulting compact piece, and glue by a finite maximum
with a barrier minorant. Thus Theorem~\ref{thm:continuity} requires neither a
Runge pair nor a pseudoconvex ambient neighbourhood; holomorphic convexity enters only in the later hull results, where it is
actually needed.

Approximation requires a separate argument. The example in
\cite[Proposition~3.1]{DieuPB2007} shows on the unit ball that pointwise
convergence of continuous boundary data is not sufficient to force convergence
of the corresponding Perron--Bremermann envelopes. Our nonmonotone result uses
a direct pluripolar perturbation of continuous minorants. It identifies an
intrinsic neighbourhood pluripolar hull of the singular boundary set as the
possible exceptional locus; the complete-pluripolar formulation follows
immediately as a special case. Jensen measures enter only later, as a separate
consequence of the density theorem. The same framework yields stability under
simultaneous nonmonotone variations of the domain and obstacle, and uniqueness
once the competing maximal function is kept below the same obstacle.

These results are complementary to other approaches to unbounded envelopes.
Nilsson--Wikstr\"om \cite{NW} and Nilsson \cite{Nilsson} use quasiboundedness,
Jensen measures and majorant conditions, while Dieu--Long--Hieu
\cite{DLH2026} treat nearly continuous, possibly unbounded data on quasi
$B$-regular domains. Our hypotheses are instead local and geometric, expressed
through a deformation and barriers on its boundary cluster set. No implication
between these frameworks is asserted; in particular, the present paper remains
in $\mathbb C^n$.

	We now fix notation. Let $\PSH(\Omega)$ denote the cone of
	plurisubharmonic functions, allowing the constant function $-\infty$. A
	continuous function with values in $(-\infty,+\infty]$ is always understood
	with the extended-real topology. Thus $\varphi(z)\to+\infty$ at a point
	where $\varphi=+\infty$, and we put $e^{-\infty}\coloneqq0$. Unless an
	extended-real range is specified, $C(X)$ denotes the real-valued continuous
	functions on $X$. Norms on $\mathbb C^n$ are Euclidean, and $dV$ denotes
	Euclidean volume on $\mathbb C^n\simeq\mathbb R^{2n}$.
	
	For $f\in C(\partial\Omega)$, define
	\begin{equation}\label{eq:PB}
		\PB_{f,\Omega}(z)\coloneqq\sup\{v(z):v\in\PSH(\Omega),\ v^*\le f
		\text{ on }\partial\Omega\},\qquad z\in\Omega,
	\end{equation}
	where $v^*(\xi)\coloneqq\limsup_{\Omega\ni z\to\xi}v(z)$. A continuous
	plurisubharmonic barrier at $\xi\in\partial\Omega$ is a function
	$\theta_\xi\in\PSH(\Omega)\cap C(\overline\Omega)$ such that
	$\theta_\xi(\xi)=0$ and $\theta_\xi<0$ on
	$\overline\Omega\setminus\{\xi\}$. We call it a \emph{neighbourhood
		barrier} if it is the restriction of a continuous plurisubharmonic function
	on an open neighbourhood of $\overline\Omega$.
	
	Before introducing the general envelope, it is useful to record a simple obstruction that explains why some restriction on the blow-up set is unavoidable.
	\begin{introproposition}
		Suppose that $u\in\PSH(\Om)\cap L^\infty_{\mathrm{loc}}(\Om)$ is bounded from
		below and $(\ddc u)^n=0$. Put
		\[
		A_u\coloneqq\{\xi\in\partial\Om:\lim_{\Om\ni z\to\xi}u(z)=+\infty\}.
		\]
		No point in the relative interior of $A_u$ admits a continuous
		plurisubharmonic barrier.
	\end{introproposition}
	In particular, the blow-up set of such a function cannot contain a relatively
	open boundary piece of a $B$-regular domain.
	
	This obstruction suggests keeping the singular boundary set visible throughout the argument. We therefore work with a superharmonic obstacle on the closure and use its plurisubharmonic minorant as the distinguished solution.
	For the rest of the paper, unless a different hypothesis is stated, assume
	\begin{equation}\label{eq:standing}
		\begin{gathered}
			\ph:\clOm\longrightarrow(-\infty,+\infty]\text{ is continuous},\\
			\ph\text{ is superharmonic on }\Om,\qquad
			\ph\not\equiv+\infty\text{ on }\Om.
		\end{gathered}
	\end{equation}
	Superharmonicity is meant in the ordinary real sense. In particular,
	$\ph\in L^1_{\mathrm{loc}}(\Om)$. Compactness gives a finite lower bound.
	Write
	\begin{equation}\label{eq:G}
		\begin{gathered}
			A\coloneqq\{z\in\clOm:\ph(z)=+\infty\},\\
			G_{\ph,\Om}(z)\coloneqq\sup\{v(z):v\in\PSH(\Om),\ v\le\ph\text{ on }\Om\}.
		\end{gathered}
	\end{equation}
	\begin{theorem}\label{thm:basic}
		Under \eqref{eq:standing}, the following conclusions hold.
		\begin{enumerate}[label=\textup{(\roman*)}]
			\item $G_{\ph,\Om}\in\PSH(\Om)\cap L^\infty_{\mathrm{loc}}(\Om)$, it is
			bounded from below, and $(\ddc G_{\ph,\Om})^n=0$ on $\Om$.
			\item If $\ph|_\Om$ is not plurisubharmonic, then $G_{\ph,\Om}<\ph$ on $\Om$.
			\item At every boundary point $\xi$ admitting a continuous plurisubharmonic
			barrier, $G_{\ph,\Om}(z)\to\ph(\xi)$ as $\Om\ni z\to\xi$.
		\end{enumerate}
	\end{theorem}
	Theorem~\ref{thm:basic} also explains the connection with pluricomplex
	Green functions. Let $f$ be a nonzero holomorphic function near
	$\overline\Omega$ defining a reduced divisor $A_f$. The Green function
	$G_{A_f,\Omega}$ associated with the divisor has the residual
	representation
	\begin{equation}\label{eq:greenidentity}
		G_{A_f,\Omega}=G_{-\log|f|,\Omega}+\log|f|
		\quad\text{on }\Omega\setminus A_f.
	\end{equation}
	Indeed, adding $\log|f|$ to an admissible minorant for
	$G_{-\log|f|,\Omega}$ gives a Green-function competitor, while
	$G_{A_f,\Omega}-\log|f|$ extends as a locally bounded plurisubharmonic
	function across $A_f\cap\Omega$ and is dominated by $-\log|f|$; see
	\cite[Proposition 2.1]{DiGreen} and \cite{LS}. Thus continuity of the
	residual envelope gives the corresponding extended-real interior
	continuity of the Green function. The erratum \cite{LSE} records a gap in
	the earlier general almost-continuity statement for such Green functions.
	
	The residual identity above also explains why the singular boundary set must be
	tracked together with the geometry of the domain. We now formulate the
	deformation condition used in the continuity theorem.
	\begin{definition}\label{def:deformation}
		A neighbourhood deformation of $\clOm$ is a family $\Phi_t$, $0\le t<t_0$,
		of biholomorphic maps defined on one fixed open neighbourhood $N$ of $\clOm$,
		with $\Phi_0=\mathrm{id}$ and $\Phi_t\to\mathrm{id}$ locally uniformly on $N$.
		The family is continuous in $(t,z)$. For each $z\in\Om$ we also require that
		$\Phi_t^{-1}(z)$ is defined for sufficiently small $t$ and is real analytic
		in $t$ near $0$. The identity curve is allowed.
	\end{definition}
	Write $D_t\coloneqq\Phi_t(\Om)$. The maps are defined near $\clOm$, whereas the
	obstacle is defined only on $\clOm$. No inclusion between distinct deformed
	domains, or between $D_t$ and $\Om$, is assumed.
	
	\begin{definition}\label{def:cluster}
		The boundary cluster set of the deformation is
		\begin{equation}\label{eq:cluster}
			S_\Phi\coloneqq\{\xi:\exists\,t_j\downarrow0,\ 
			y_j\in\clOm\cap\Phi_{t_j}(\partial\Om),\ y_j\to\xi\}.
		\end{equation}
		It is a compact subset of $\partial\Om$ and may be empty.
	\end{definition}
	We use the following classes of finite continuous minorants:
	\begin{align}
		\cclass&\coloneqq\{v\in\PSH(\Om)\cap C(\clOm):v\le\ph\text{ on }\clOm\},
		\label{eq:Cclass}\\
		\nclass&\coloneqq\bigl\{v|_{\clOm}:W\supset\clOm\text{ open},\notag\\
		&\hspace{4.4em}v\in\PSH(W)\cap C(W),\ v\le\ph\text{ on }\clOm\bigr\}.
		\label{eq:Nclass}
	\end{align}
	\begin{theorem}\label{thm:continuity}
		Assume \eqref{eq:standing} and let $\Phi_t$ be a neighbourhood deformation
		as in Definition~\ref{def:deformation}. Write $B\coloneqq A\cap\partial\Om$ and
		$D_t\coloneqq\Phi_t(\Om)$. Assume the following two conditions.
		\begin{enumerate}[label=\textup{(D\arabic*)}]
			\item\label{cond:obstacle}
			Whenever $t_j\downarrow0$ and $x_j\in\Om$ tend to a point of $B$, with
			$y_j\coloneqq\Phi_{t_j}(x_j)\in\clOm\setminus A$, one has
			\begin{equation}\label{eq:compatibility}
				\limsup_{j\to\infty}\bigl(\ph(x_j)-\ph(y_j)\bigr)\le0.
			\end{equation}
			Here $\ph(y_j)$ is finite; if $x_j\in A$, the corresponding difference is
			$+\infty$. Thus no expression $+\infty-(+\infty)$ occurs.
			\item\label{cond:barrier}
			$S_\Phi\cap A=\varnothing$, and every point of $S_\Phi$ admits a continuous
			plurisubharmonic barrier.
		\end{enumerate}
		Then
		\begin{equation}\label{eq:density}
			G_{\ph,\Om}(z)=\sup_{v\in\cclass}v(z),\qquad z\in\Om,
		\end{equation}
		and $G_{\ph,\Om}$ is continuous on $\Om$. Moreover, if
		$f_j\in C(\partial\Om)$ increases pointwise to $\ph$ on the whole boundary,
		then
		\[
		\PB_{f_j,\Om}\uparrow G_{\ph,\Om}\quad\text{on }\Om,
		\]
		and the convergence is locally uniform. If the barriers at the points of
		$S_\Phi$ are neighbourhood barriers, then \eqref{eq:density} remains valid
		with $\nclass$ in place of $\cclass$.
		No inclusion between distinct $D_t$, or between $D_t$ and $\Om$, is
		required. The ambient neighbourhood need not be pseudoconvex or Runge.
	\end{theorem}
	The proof of Theorem~\ref{thm:continuity} rests on two geometric-analytic
	ingredients. First, the transported envelope $G_{\ph,\Om}\circ\Phi_t^{-1}$
	satisfies a uniform one-sided estimate on $D_t\cap\clOm$; the obstacle is
	never evaluated outside the fixed closure. Second, Lemma~\ref{lem:compactcuts}
	produces a compact cut whose artificial boundary lies in the barrier region,
	so a regularization of the transported envelope can be glued to a barrier
	minorant by a finite maximum. Once the density formula \eqref{eq:density} is
	available, nonmonotone boundary approximation is obtained separately by a
	pluripolar perturbation. Lemma~\ref{lem:radiuscriterion} reformulates the
	one-sided compatibility condition in terms of $e^{-\ph}$ and includes the zero
	fibres of logarithmic obstacles. In the homogeneous hyperplane situation of
	\cite{DiGreen}, the corresponding control is supplied by the identity
	$f(tz)=t^d f(z)$ and by scalar dilations. Condition (D1) isolates the part of
	that calculation which is actually needed for continuity and makes it
	available for nonhomogeneous obstacles and nonradial deformations.
	
	Once neighbourhood minorants are available, the same approximation mechanism can be combined with the classical disk-fibre construction. We now return to holomorphic convexity and obtain a precise hull representation.
	For a compact set $K$ in a pseudoconvex domain $U$, let
	\[
	\hull K U\coloneqq\{z\in U:|h(z)|\le\max_K|h|\text{ for every }h\in\OO(U)\}.
	\]
	We say that $\clOm$ is $\OO(U)$-convex if
	$\hull{\clOm}U=\clOm$.
	\begin{theorem}\label{thm:hull}
		Assume the hypotheses of Theorem~\ref{thm:continuity}. Suppose, in addition, that $U$ is a
		pseudoconvex domain containing $\clOm$, that $\clOm$ is $\OO(U)$-convex, and
		that every point of $S_\Phi$ admits a neighbourhood barrier. Set
		\begin{align*}
			K&\coloneqq\{(z,w):z\in\clOm,\ |w|\le e^{-\ph(z)}\},\\
			K_\partial&\coloneqq\{(z,w):z\in\partial\Om,\ |w|\le e^{-\ph(z)}\}.
		\end{align*}
		Then
		\begin{equation}\label{eq:hullidentity}
			\hull K{U\times\C}\cap(\Om\times\C)
			=\{(z,w):z\in\Om,\ |w|\le e^{-G_{\ph,\Om}(z)}\},
		\end{equation}
		and
		\begin{equation}\label{eq:fullhullequality}
			\hull K{U\times\C}=\hull{K_\partial}{U\times\C}.
		\end{equation}
	\end{theorem}
	The inclusion from the Perron envelope into the hull is the
	coefficient-and-powers argument of \cite[Lemmas 2.5--2.6]{DiGreen}. In the
	2005 paper the hull construction is part of the engine used to prove
	continuity. Here the logical direction is reversed: Theorem~\ref{thm:continuity}
	first produces neighbourhood minorants by local deformation and gluing, and
	those minorants then give the reverse hull inclusion. This keeps the strongest
	part of the Green-function idea in \cite{DiGreen} while avoiding any
	holomorphic-convexity assertion for a connected component selected from an
	intersection of moving domains.
	
	The continuity and hull theorems already give concrete boundary-value consequences. A first family comes from logarithmic singularities generated by homogeneous linear forms.
	\begin{corollary}\label{cor:homogeneous}
		Let $\Om$ be a bounded strictly convex domain containing $0$. For each
		$1\le i\le k$, let $f_{i,1},\ldots,f_{i,m_i}$ be homogeneous linear
		polynomials, with at least one $f_{i,j}$ nonzero, and let $\alpha_i>0$. Put
		\[
		\Psi_0(z)\coloneqq-\sum_{i=1}^k
		\log\left(\sum_{j=1}^{m_i}|f_{i,j}(z)|^{\alpha_i}\right).
		\]
		For every real-valued $h\in C(\partial\Om)$ there exists
		$u\in\PSH(\Om)\cap C(\Om)$, bounded from below, such that
		\[
		(\ddc u)^n=0,\qquad
		\lim_{\Om\ni z\to\xi}u(z)=\Psi_0(\xi)+h(\xi)
		\quad(\xi\in\partial\Om).
		\]
		More precisely, $u=G_{\Psi_0+H,\Om}$, where $H$ is the harmonic extension of
		$h$ to $\Om$, continuous on $\clOm$.
	\end{corollary}
	Here the continuous correction $h$ is finite-valued; this is used in the construction of the superharmonic obstacle.
	
	\begin{corollary}\label{cor:minimal}
		Assume the hypotheses of Theorem~\ref{thm:hull}, let $m\coloneqq\min_{\partial\Om}\ph<+\infty$, and put
		$T\coloneqq\{\xi\in\partial\Om:\ph(\xi)=m\}$. Then
		\[
		\{z\in\Om:G_{\ph,\Om}(z)=m\}=\Om\cap\hull T U.
		\]
	\end{corollary}
	
	The method also remains stable when the obstacle and the domain vary simultaneously, provided convergence of the domains is controlled from both sides. The following statement makes this precise without imposing any monotonicity on the sequence of domains.
	\begin{corollary}\label{cor:stability}
		Assume the hypotheses of Theorem~\ref{thm:continuity}, and suppose every
		point of $S_\Phi$ admits a neighbourhood barrier. Assume that $\ph$ extends
		to a continuous superharmonic function
		$\ph:U\to(-\infty,+\infty]$ on an open neighbourhood $U$ of $\clOm$.
		Let $V\Subset U$ be an open neighbourhood of $\clOm$, and let
		$\Om_j\Subset V$ be bounded domains. For every $j$, let
		$\ph_j:U\to(-\infty,+\infty]$ be continuous and superharmonic, not identically
		$+\infty$. Assume
		\begin{enumerate}[label=\textup{(\alph*)}]
			\item $e^{-\ph_j}\to e^{-\ph}$ uniformly on $\overline V$;
			\item every compact subset of $\Om$ lies in $\Om_j$ for all sufficiently large $j$;
			\item if $j_\nu\to\infty$, $z_\nu\in\overline{\Om_{j_\nu}}$, and $z_\nu\to z$,
			then $z\in\clOm$.
		\end{enumerate}
		Then $G_{\ph_j,\Om_j}\to G_{\ph,\Om}$ locally uniformly on $\Om$.
	\end{corollary}
	The counterexample in Example~\ref{ex:stability} shows that the outer condition (c) cannot in general be dropped.
	
	Continuity and approximation also lead to uniqueness once the competing maximal function is required to stay below the same obstacle. This fixed majorant is essential when the boundary values are unbounded.
	\begin{theorem}\label{thm:uniqueness}
		Assume the hypotheses of Theorem~\ref{thm:continuity}. Suppose that $A\cap\partial\Om$ is
		contained in a complete pluripolar set on a connected neighbourhood of $\clOm$, and that
		every $\xi\in\partial\Om\setminus A$ admits a continuous plurisubharmonic
		barrier. Then $G_{\ph,\Om}$ is the unique function $u$ satisfying
		\begin{gather*}
			u\in\PSH(\Om),\qquad \inf_\Om u>-\infty,\qquad u\le\ph,\\
			(\ddc u)^n=0,\qquad
			\lim_{\Om\ni z\to\xi}u(z)=\ph(\xi)
			\quad(\xi\in\partial\Om\setminus A).
		\end{gather*}
	\end{theorem}
	A common boundary trace without the fixed majorant need not give
	uniqueness; see the example in Section~\ref{sec:hullproofs}.
	
	For applications, the deformation hypothesis must be checked in terms of the geometry of the boundary and the singular set. The next results provide a differential criterion for avoiding the singular boundary and an analytic criterion for preserving logarithmic singularities.
	\begin{proposition}\label{prop:transverse}
		Suppose $\Om=\{\rho<0\}$ near $\clOm$, where $\rho$ is $C^1$ on a
		neighbourhood of $\clOm$, and let $\Phi_t$ be a neighbourhood deformation.
		Assume that its inverse is defined on a fixed neighbourhood of $\clOm$ for
		small $t$ and is jointly $C^1$ in $(z,t)$. If
		\begin{equation}\label{eq:negativeinverse}
			\Rea\sum_{i=1}^n\rho_{z_i}(\xi)
			\left.\frac{\partial (\Phi_t^{-1})_i}{\partial t}(\xi)\right|_{t=0}<0
			\quad(\xi\in A\cap\partial\Om),
		\end{equation}
		then $S_\Phi\cap A=\varnothing$. In fact a fixed relative neighbourhood of
		$A\cap\partial\Om$ in $\clOm$ lies in $\Phi_t(\Om)$ for every sufficiently
		small positive $t$.
	\end{proposition}
	Thus the criterion has a simple geometric interpretation: the inverse velocity has negative normal component, or equivalently the forward velocity points outward, along the singular boundary.
	
	\begin{lemma}\label{lem:divisor}
	Let $f$ be a nonzero reduced holomorphic function on a neighbourhood of
	$\clOm$, and let $\Phi_t$ be jointly holomorphic in $(z,t)$ for complex $t$
	near $0$, with $\Phi_0=\mathrm{id}$. Suppose that, for all sufficiently small
	real $t\ge0$, the divisor is preserved on a fixed neighbourhood:
	\[
	\Phi_t(\{f=0\})\subset\{f=0\}.
	\]
	After restricting the neighbourhood around $\clOm$ and taking $t$ small, the
	following assertions hold.
	\begin{enumerate}[label=\textup{(\roman*)}]
		\item There is a holomorphic unit $\theta_t$ such that
		\begin{equation}\label{eq:unitfactor}
			f\circ\Phi_t=\theta_t f,
			\qquad
			\theta_t\ne0,
			\qquad
			\theta_t\longrightarrow1
			\quad\text{uniformly on }\clOm.
		\end{equation}
		\item Consequently,
		\[
		\sup_{z\in\clOm\setminus\{f=0\}}
		\big|\log|f(\Phi_t(z))|-\log|f(z)|\big|
		\longrightarrow0.
		\]
	\end{enumerate}
\end{lemma}
	The divisor criterion becomes particularly concrete for disjoint holomorphic graphs. Writing $\pi(z',z_n)\coloneqq z'$, we first construct deformations that preserve each graph.
	\begin{proposition}\label{prop:graphs}
		Let $\Om\Subset\Om'\Subset U\subset\C^n$, where $U$ is open. Suppose
		$f_1,\ldots,f_k\in\OO(\pi(U))$ and
		$f_i(z')\ne f_j(z')$ for every $i\ne j$ and $z'\in\pi(U)$.
		Given a holomorphic map $\Psi:\pi(U)\to\C^{n-1}$, there is a jointly
		holomorphic function $\Theta(z,t)$ near $\overline{\Om'}\times\{0\}$ such that
		\[
		\Phi_t(z',z_n)\coloneqq(z'+t\Psi(z'),\ z_n+t\Theta(z',z_n,t))
		\]
		is a neighbourhood deformation of $\clOm$ and
		\[
		\Phi_t(z',f_j(z'))=(z'+t\Psi(z'),\ f_j(z'+t\Psi(z')))
		\]
		wherever these expressions are defined near $\overline{\Om'}$.
	\end{proposition}
	
	\begin{proposition}\label{prop:graphcontinuity}
		Let $U\subset\C^n$ be a domain admitting a $C^2$ strictly plurisubharmonic
		exhaustion $\rho$. Let $\alpha$ be a regular value of $\rho$, and let $\Om$
		be a connected component of $\{z\in U:\rho(z)<\alpha\}$. Suppose that
		$f_1,\ldots,f_k\in\OO(\pi(U))$ are pairwise distinct at every point of
		$\pi(U)$, and let $\Psi:\pi(U)\to\C^{n-1}$ be holomorphic. Set
		\[
		f(z)\coloneqq\prod_{j=1}^k(z_n-f_j(z')),\qquad \ph\coloneqq-\log|f|,\qquad A_j\coloneqq\{z_n=f_j(z')\}.
		\]
		For $\xi\in A_j\cap\partial\Om$, define
		\begin{equation}\label{eq:graphnormal}
			h_j(\xi)\coloneqq\Rea\sum_{i=1}^{n-1}\Psi_i(\xi')
			\left(\rho_{z_i}(\xi)+\frac{\partial f_j}{\partial z_i}(\xi')
			\rho_{z_n}(\xi)\right).
		\end{equation}
		If all these quantities are positive, or all are negative, then
		$G_{\ph,\Om}$ is continuous on $\Om$. The conclusions of
		Theorems~\ref{thm:hull} and~\ref{thm:uniqueness} also hold with this obstacle
		and ambient domain $U$.
	\end{proposition}
	Example~\ref{ex:curved} later illustrates this criterion for two genuinely curved, disjoint graphs.
	
	The rest of the paper is organized as follows. Section~\ref{sec:aux}
	collects the auxiliary lemmas, Section~\ref{sec:proofcontinuity} proves the
	continuity and approximation results, and Section~\ref{sec:hullproofs} treats
	hulls, stability and uniqueness. Sections~\ref{sec:geometry} and
	\ref{sec:scope} give geometric applications and the extension to holomorphic
	tuples.
	
	\section{Preliminaries and approximation lemmas}\label{sec:aux}

	We begin with the comparison principle that will be used repeatedly to pass
	from boundary inequalities to the obstacle envelope.

	\begin{lemma}\label{lem:comparison}
	Let $u\in\PSH(\Om)\cap L^\infty_{\mathrm{loc}}(\Om)$ be maximal and bounded from below, and let
	$v\in\PSH(\Om)$.
	\begin{enumerate}[label=\textup{(\roman*)}]
		\item If
		\[
		\limsup_{\Om\ni z\to\xi}\bigl(v(z)-u(z)\bigr)\le0,
		\qquad \xi\in\partial\Om,
		\]
		then
		\[
		v\le u\qquad\text{on }\Om.
		\]
		\item Under \eqref{eq:standing}, if
		\[
		f\in C(\partial\Om),\qquad f\le\ph\quad\text{on }\partial\Om,
		\]
		then
		\begin{equation}\label{eq:PBbound}
			\PB_{f,\Om}\le G_{\ph,\Om}\le\ph
			\qquad\text{on }\Om.
		\end{equation}
	\end{enumerate}
\end{lemma}
	\begin{proof}
		Fix $\delta>0$. For each $\xi\in\partial\Om$, the boundary limsup
		assumption gives a neighbourhood $V_\xi$ of $\xi$ such that
		\[
			v\le u+\delta
			\qquad\text{on }\Om\cap V_\xi.
		\]
		A finite subcover therefore gives a relative boundary neighbourhood $V$ on
		which this inequality holds. Choose an open set $D\Subset\Om$ with
		$\Om\setminus V\subset D$; then $\partial D\subset V$. Since $u$ is maximal, the finite translate $u+\delta$ is again maximal,
		and the defining comparison property gives $v\le u+\delta$ on $D$. The same inequality already holds on $\Om\setminus D$. Letting
		$\delta\downarrow0$ proves \textup{(i)}.
		For the second, an admissible $v$ in \eqref{eq:PB} is bounded above by
		$\max_{\partial\Om}f$, by the subharmonic maximum principle. The function
		$v-\ph$ is subharmonic. At finite boundary points its upper limit is at most
		zero; at a point of $A$ it tends to $-\infty$. Thus $v\le\ph$ on $\Om$.
		Taking the supremum over $v$ proves \eqref{eq:PBbound}.
	\end{proof}

	The comparison principle immediately turns a boundary barrier into a global
	plurisubharmonic minorant. The next lemma records the uniform version needed
	near a compact set of boundary points.

	\begin{lemma}\label{lem:barrierminorants}
	Suppose \eqref{eq:standing}, and let
	\[
	S\subset\partial\Om\setminus A
	\]
	be compact. Assume that every point of $S$ admits a continuous
	plurisubharmonic barrier. Then, for every $\eps>0$:
	\begin{enumerate}[label=\textup{(\roman*)}]
		\item there exist $u_S\in\cclass$ and a relative neighbourhood $W$ of $S$
		in $\clOm$ such that
		\[
		u_S>\ph-\eps\qquad\text{on }W;
		\]
		moreover, $\overline W$ may be chosen disjoint from $A$ and from any
		prescribed compact subset of $\Om$;
		\item if the barriers are neighbourhood barriers, then $u_S$ may be chosen in
		$\nclass$.
	\end{enumerate}
\end{lemma}
	\begin{proof}
		Choose a constant $m\le\inf_{\clOm}\ph$. For $\xi\in S$, put
		\[
		c_\xi\coloneqq\ph(\xi)-\eps/4.
		\] Continuity makes $\ph>c_\xi$ on a boundary
		neighbourhood of $\xi$. Away from a smaller such neighbourhood, a barrier
		$\theta_\xi$ is bounded above by a negative constant. Consequently, for a
		sufficiently large $M_\xi>0$, the function
		\[
		u_\xi\coloneqq c_\xi+M_\xi\theta_\xi
		\]
		is at most $\ph$ on $\partial\Om$.
		Apply Lemma~\ref{lem:comparison}\textup{(ii)} with the boundary datum
		$f=u_\xi|_{\partial\Om}$. Since $u_\xi$ itself is an admissible
		Perron--Bremermann competitor for this datum, we obtain
		$u_\xi\le\PB_{f,\Om}\le G_{\ph,\Om}\le\ph$ on $\Om$. Since $u_\xi(\xi)=\ph(\xi)-\eps/4$, the
		inequality $u_\xi>\ph-\eps$ persists on a relative neighbourhood of $\xi$
		in $\clOm$. A finite maximum of these functions and $m$ supplies $u_S$.
		The same finite maximum is defined and plurisubharmonic on a neighbourhood
		of $\clOm$ when all the chosen barriers extend there. Shrinking $W$ gives
		the last separation properties.
	\end{proof}
	
	We next isolate two elementary consequences of the deformation hypotheses.
	The first rewrites the one-sided obstacle condition in terms of the radius
	function $e^{-\ph}$, while the second locates the part of the moving boundary
	that can meet the fixed closure.
	
	\begin{lemma}\label{lem:radiuscriterion}
	Put
	\[
	B\coloneqq A\cap\partial\Om,
	\qquad
	r\coloneqq e^{-\ph}\quad\text{on }\clOm.
	\]
	Condition \textup{(D1)} in Theorem~\ref{thm:continuity} is equivalent to the
	following local condition: for every $\eps>0$ there are $\delta>0$ and an open
	neighbourhood $V_\eps$ of $B$ such that
	\begin{equation}\label{eq:radiuscriterion}
		\begin{gathered}
			r(\Phi_t(x))\le e^\eps r(x),\\
			0<t<\delta,\qquad
			x\in\Om\cap V_\eps,\qquad
			\Phi_t(x)\in\clOm.
		\end{gathered}
	\end{equation}
	If $B=\varnothing$, the condition is vacuous. At a zero of $r$, the
	inequality in \eqref{eq:radiuscriterion} is interpreted literally.
\end{lemma}
	\begin{proof}
		Assume (D1) and suppose that \eqref{eq:radiuscriterion} fails for some
		$\eps>0$ for every choice of $\delta,V_\eps$. We may choose a strictly
		decreasing sequence $t_j\to0$ and points $x_j\in\Om$ satisfying
		\[
		\dist(x_j,B)<1/j,\qquad
		y_j=\Phi_{t_j}(x_j)\in\clOm,\qquad
		r(y_j)>e^\eps r(x_j).
		\] Passing to a subsequence gives $x_j\to a\in B$.
		The strict inequality implies $r(y_j)>0$, so $\ph(y_j)$ is finite. If
		$r(x_j)>0$, it is equivalent to
		\[
		\ph(x_j)-\ph(y_j)>\eps.
		\] If $r(x_j)=0$, the latter difference is
		$+\infty$. Both alternatives contradict (D1).
		
		Conversely, consider the sequences in (D1). For each fixed $\eps>0$,
		their points eventually satisfy \eqref{eq:radiuscriterion}. Since
		$r(y_j)>0$, this also forces $r(x_j)>0$ for all sufficiently large $j$.
		Taking logarithms gives $\ph(x_j)-\ph(y_j)\le\eps$ eventually. Let
		$\eps\downarrow0$ to obtain (D1). This proves the equivalence.
	\end{proof}
	
	\begin{lemma}\label{lem:compactcuts}
	Let $\Phi_t$ be as in Definition~\ref{def:deformation} and set
	\[
	D_t\coloneqq\Phi_t(\Om).
	\]
	Then:
	\begin{enumerate}[label=\textup{(\roman*)}]
		\item for every $L\Subset\Om$,
		\[
		L\subset D_t
		\]
		for all sufficiently small $t$;
		\item for every relative open neighbourhood $W$ of $S_\Phi$ in $\clOm$,
		\[
		\partial D_t\cap\clOm\subset W
		\]
		for all sufficiently small $t$;
		\item if $\partial D_t\cap\clOm\subset W$, then
		\[
		L_t\coloneqq(\clOm\setminus W)\cap D_t\Subset D_t.
		\]
		Moreover, for every open set $Q$ satisfying
		\[
		L_t\subset Q\Subset D_t,
		\]
		one has
		\[
		\partial Q\cap\clOm\subset W.
		\]
	\end{enumerate}
	No connectedness, pseudoconvexity or smoothness of $Q$ is required.
\end{lemma}
	\begin{proof}
		For (i), choose a common radius $a>0$ so that the balls of radius $2a$
		about points of $L$ lie in $\Om$. Uniform convergence of $\Phi_t$ to the
		identity on a slightly larger compact set, and Cauchy's estimates, give
		uniform convergence of their complex derivatives to the identity there.
		For $y\in L$ and small $t$, the map
		\[
		T_{t,y}(x)\coloneqq y-(\Phi_t(x)-x)
		\]
		is a contraction of the closed ball $\overline{B(y,a)}$ into itself.
		Its fixed point belongs to $\Om$ and solves $\Phi_t(x)=y$. The choices are
		uniform for $y\in L$, proving (i).
		
		If (ii) fails, choose $t_j\downarrow0$ and
		$y_j\in(\partial D_{t_j}\cap\clOm)\setminus W$. A convergent subsequence
		has a limit in $S_\Phi$. Since $\clOm\setminus W$ is closed in the compact
		set $\clOm$, its limit also lies outside $W$, a contradiction. The identity
		$\partial D_t=\Phi_t(\partial\Om)$ follows from the fact that the maps
		are homeomorphisms on a neighbourhood of the closure.
		
		For (iii), any limit of points of $L_t$ lies in
		$(\clOm\setminus W)\cap\overline{D_t}$. Such a limit cannot lie in
		$\partial D_t$, by the assumed boundary inclusion. Hence $L_t$ is closed
		and bounded, and is contained in the open set $D_t$. If
		$x\in\partial Q\cap\clOm$ were outside $W$, then
		$x\in(\clOm\setminus W)\cap D_t=L_t\subset Q$, which is impossible.
		This proves the last assertion.
	\end{proof}

	The preceding lemma supplies the geometric cut. We now record the analytic
	regularization facts that will be applied on that compact piece.

	\begin{lemma}\label{lem:regularization}
	The following two facts will be used.
	\begin{enumerate}[label=\textup{(\roman*)}]
		\item Let
		\[
		u\in\PSH(D)\cap L^\infty_{\mathrm{loc}}(D),
		\qquad L\Subset D.
		\]
		There are smooth plurisubharmonic functions $u_\nu$ on a fixed
		neighbourhood of $L$ such that
		\[
		u_\nu\downarrow u\qquad\text{on }L.
		\]
		If $\varphi:L\to(-\infty,+\infty]$ is continuous and
		\[
		u\le\varphi+\eta\qquad\text{on }L,
		\]
		then, for every $\delta>0$,
		\[
		u_\nu\le\varphi+\eta+\delta\qquad\text{on }L
		\]
		for all sufficiently large $\nu$.
		\item Let $u$ be locally bounded and plurisubharmonic near $z$, and let
		$\gamma$ be a real-analytic curve with $\gamma(0)=z$ admitting a holomorphic
		complexification near $0$. Then
		\[
		u(z)=\limsup_{t\downarrow0}u(\gamma(t)).
		\]
	\end{enumerate}
\end{lemma}
	\begin{proof}
		For (i), convolve with a nonnegative smooth radial approximate identity on
		slightly smaller open sets. Radial means of a subharmonic function are
		nondecreasing with the radius, so radii decreasing to zero give the indicated
		monotonicity. To obtain the uniform upper estimate, consider
		\[
		E_\nu\coloneqq\{x\in L:u_\nu(x)\ge\varphi(x)+\eta+\delta\}.
		\] Since $\varphi$ is
		continuous with values in $(-\infty,+\infty]$, each $E_\nu$ is compact,
		and points where $\varphi=+\infty$ do not belong to it. The sets $E_\nu$
		decrease with $\nu$ and have empty intersection because
		$u_\nu\downarrow u\le\varphi+\eta$. Hence $E_\nu=\varnothing$ for all large
		$\nu$.
		For (ii), composition with the holomorphic complexification gives a
		finite subharmonic function near $0$. The assertion is immediate for a
		constant curve. A real interval is not thin at its endpoint. The
		one-variable characterization of thinness therefore gives the equality;
		see \cite[Chapter 4]{Kl}. Upper semicontinuity supplies the opposite
		inequality as well.
	\end{proof}

	Regularization handles the local construction. We now turn to the
	holomorphic-hull side of the argument. The following
	coefficient estimate gives the basic disk-fibre inclusion.

	\begin{lemma}\label{lem:Cauchyhull}
	Let $D\Subset U$, where $U$ is pseudoconvex.
	\begin{enumerate}[label=\textup{(\roman*)}]
		\item For $f\in C(\partial D)$, put
		\[
		L_f\coloneqq\{(z,w):z\in\partial D,\ |w|\le e^{-f(z)}\}.
		\]
		Then
		\[
		\{(z,w):z\in D,\ |w|\le e^{-\PB_{f,D}(z)}\}
		\subset\hull{L_f}{U\times\C}.
		\]
		\item Suppose $\psi:\overline D\to(-\infty,+\infty]$ is continuous and
		bounded below, and $G_{\psi,D}$ is finite. Put
		\[
		K_\psi\coloneqq\{(z,w):z\in\overline D,\ |w|\le e^{-\psi(z)}\}.
		\]
		Then
		\[
		\{(z,w):z\in D,\ |w|\le e^{-G_{\psi,D}(z)}\}
		\subset\hull{K_\psi}{U\times\C}.
		\]
	\end{enumerate}
\end{lemma}
	\begin{proof}
		The proof is the coefficient-and-powers argument of
		\cite[Lemma 2.5]{DiGreen}. Let $H\in\OO(U\times\C)$ with
		$\|H\|_{L_f}\le1$ and write
		$H(z,w)=\sum_{m=0}^{\infty}H_m(z)w^m$. Cauchy's estimate gives
		\[
		|H_m(z)|\le e^{mf(z)}\quad(z\in\partial D,\ m\ge1),
		\]
		and therefore
		\[
		m^{-1}\log|H_m|\le\PB_{f,D}\quad\text{on }D.
		\]
		Also $|H_0|\le1$ on $D$.
		If $r\coloneqq|w|e^{\PB_{f,D}(z)}<1$, then
		\[
		|H(z,w)|\le\sum_{m=0}^{\infty}r^m=(1-r)^{-1}.
		\]
		Apply the same inequality to $H^N$, whose norm on $L_f$ is at most one,
		and take $N$th roots. Letting $N\to\infty$ gives $|H(z,w)|\le1$.
		Closure gives the limiting fibre circle. Scaling $H$ (or first dividing by
		$\|H\|+\delta$) establishes the hull inclusion.
		
		For (ii), apply Cauchy's estimate over every point of $D$ where $\psi$ is finite.
		It gives $m^{-1}\log|H_m|\le\psi$ on $D$; at points where $\psi=+\infty$ the
		inequality is automatic. By the definition of the obstacle envelope, the
		left side is at most $G_{\psi,D}$. The zero section
		$\overline D\times\{0\}$ lies in $K_\psi$, so the normalization
		$\|H\|_{K_\psi}\le1$ also gives $|H_0|\le1$ on $D$.
		The same powers argument now proves (ii). Neither assertion requires
		holomorphic convexity of $\overline D$.
	\end{proof}

	For the reverse hull inclusion we use the standard agreement between
	holomorphic and plurisubharmonic hulls in a Stein domain.

	\begin{lemma}\label{lem:localhull}
	Let $K$ be a compact subset of a pseudoconvex domain $U$, and let $u$ be
	plurisubharmonic on a neighbourhood of $\hull K U$. Then
	\[
	u\le\sup_K u
	\qquad\text{on }\hull K U.
	\]
\end{lemma}
	\begin{proof}
		Since a pseudoconvex domain in $\C^n$ is Stein, holomorphic and
		plurisubharmonic hulls agree in $U$; see \cite[Chapter 4]{Ho}. Thus the
		defining holomorphic inequalities for $\hull K U$ are equivalent to the
		plurisubharmonic hull inequalities for plurisubharmonic functions defined
		on $U$. To apply this to the present $u$, which is only defined near
		$\hull K U$, choose a Runge Stein neighbourhood $V$ of the holomorphically
		convex compact set $\hull K U$ with
		$\hull K U\Subset V\Subset U$ and $V$ contained in the domain of $u$.
		Holomorphically convex compact sets in a Stein manifold admit a basis of
		Runge Stein neighbourhoods; see, for example, the Stein-neighbourhood and
		Oka--Weil results collected in \cite[Chapter~4]{Ho}. Oka--Weil approximation
		on $\hull K U$ shows that the holomorphic hull of $K$ computed in $V$ is still
		$\hull K U$. Applying the plurisubharmonic hull characterization in $V$
		gives the assertion; compare also \cite[Section~II]{DiGreen}.
	\end{proof}

	To pass from fixed disk bundles to varying ones, we also need a compactness
	principle for holomorphic hulls.

	\begin{lemma}\label{lem:hullcompactness}
	Let $K_j$ be compact subsets of one fixed compact set $L\Subset U$, where
	$U$ is pseudoconvex.
	\begin{enumerate}[label=\textup{(\roman*)}]
		\item If every cluster point of the sequence of sets $K_j$ belongs to a
		compact set $K\Subset U$, then every cluster point of $\hull{K_j}U$ belongs
		to $\hull K U$.
		\item If
		\[
		K_j\downarrow K,
		\]
		then
		\[
		\hull{K_j}U\downarrow\hull K U.
		\]
	\end{enumerate}
\end{lemma}
	\begin{proof}
		Since the pseudoconvex domain $U\subset\C^n$ is Stein, it is
		holomorphically convex; hence the hull $\hull L U$ of the compact set $L$
		is itself compact in $U$. All the hulls $\hull{K_j}U$ lie in this fixed
		compact set. For $H\in\OO(U)$,
		compactness and the cluster assumption give
		\[
		\limsup_{j\to\infty}\max_{K_j}|H|\le\max_K|H|.
		\] Passing to the limit in the defining
		hull inequality proves the first statement. For decreasing compacts,
		$\max_{K_j}|H|\downarrow\max_K|H|$, which proves the second statement.
	\end{proof}

	\section{The obstacle envelope and the continuity theorem}\label{sec:proofcontinuity}
	
	\begin{proof}[Proof of the preceding proposition]
		Since $u$ is locally bounded and $(\ddc u)^n=0$, the Bedford--Taylor
		criterion \cite{BT} implies that $u$ is maximal. After adding a constant, assume
		$u\ge0$. If a barrier $\theta$ exists at a point
		$\xi$ in the relative interior of $A_u$, choose a relative boundary
		neighbourhood $W$ of $\xi$ contained in $A_u$. For large $M$, the continuous
		plurisubharmonic function $v_M\coloneqq\max\{1+M\theta,0\}$ is zero near
		$\partial\Om\setminus W$ and is positive at some interior point near $\xi$.
		For every positive integer $N$,
		\[
		\limsup_{\Om\ni z\to\zeta}\bigl(Nv_M(z)-u(z)\bigr)\le0
		\quad(\zeta\in\partial\Om).
		\]
		Indeed $v_M=0$ at the relevant nonsingular boundary points, whereas $u\to+\infty$
		on $W$. Lemma~\ref{lem:comparison} gives $Nv_M\le u$ on $\Om$.
		At an interior point where $v_M>0$, this contradicts local boundedness of $u$
		as $N\to\infty$. The proposition follows.
	\end{proof}
	
	\begin{proof}[Proof of Theorem~\ref{thm:basic}]
		Write $G\coloneqq G_{\ph,\Om}$ and choose $m\le\inf_{\clOm}\ph$. Constants show that
		$G\ge m$. If $v\in\PSH(\Om)$ and $v\le\ph$, the submean inequality gives
		\begin{equation}\label{eq:meanbound}
			v(x)\le\frac{1}{|B(x,r)|}\int_{B(x,r)}\ph\,dV
			\quad\text{whenever }B(x,r)\Subset\Om.
		\end{equation}
		For $x$ in a fixed compact subset, choose one sufficiently small $r$ and
		bound the right side by the integral of $\ph^+$ on a slightly larger compact
		set divided by $|B(0,r)|$. It follows that all the admissible functions are
		locally uniformly bounded above. Hence $G^*$ is plurisubharmonic and locally
		bounded. At a finite point of $\ph$, continuity gives $G^*\le\ph$; on $A$ this
		is automatic. Thus $G^*$ is itself admissible and $G=G^*$.
		
		Let $D\Subset\Om$ and let $v\in\PSH(D)$ satisfy
		\[
		\limsup_{D\ni z\to\xi}\bigl(v(z)-G(z)\bigr)\le0
		\qquad(\xi\in\partial D).
		\]
		This difference inequality is the boundary condition in the definition of
		maximality; merely requiring $\limsup v(z)\le G(\xi)$ would not suffice
		before continuity of $G$ is known. Local boundedness of $G$ and
		compactness of $\partial D$ imply
		that $v$ is bounded above on $D$. Thus $v-\ph$ has nonpositive boundary
		upper limits at finite points of $\ph$ and tends to $-\infty$ at its
		infinite points. The subharmonic maximum principle gives $v\le\ph$ on $D$. By the standard plurisubharmonic pasting lemma, the function equal to
		$\max\{v,G\}$ on $D$ and to $G$ off $D$ is a plurisubharmonic minorant
		of $\ph$. By the definition of
		$G$ it cannot exceed $G$. Therefore $v\le G$ and $G$ is maximal.
		The Bedford--Taylor equivalence \cite{BT} proves $(\ddc G)^n=0$.
		
		The subharmonic function $G-\ph$ is nonpositive. If it is zero at an interior
		point, the strong maximum principle makes it identically zero, so that
		$\ph=G$ is plurisubharmonic. This proves (ii).
		
		For (iii), take $f_j=\min\{\ph,j\}$ on the boundary and fix a point
		$\xi$ admitting a barrier $\theta_\xi$. Given $j$ and $\eps>0$, put
		$c\coloneqq f_j(\xi)-\eps/2$. By continuity of $f_j$, one has $f_j>c$ near
		$\xi$ on the boundary. Since $\theta_\xi<0$ on the compact complement of a
		smaller boundary neighbourhood of $\xi$, a sufficiently large $M>0$ makes
		$u_{j,\xi}\coloneqq c+M\theta_\xi\le f_j$ on all of $\partial\Omega$. Thus
		$u_{j,\xi}\in\PSH(\Omega)\cap C(\overline\Omega)$ and
		$u_{j,\xi}(\xi)>f_j(\xi)-\eps$. Hence
		$u_{j,\xi}\le\PB_{f_j,\Omega}\le G$ by \eqref{eq:PBbound}.
		It follows that
		\[
		\liminf_{\Om\ni z\to\xi}G(z)\ge f_j(\xi).
		\]
		Let $j\to\infty$. The upper inequality follows from $G\le\ph$ and
		continuity of $\ph$, and is vacuous when $\ph(\xi)=+\infty$. Thus the full
		extended-real limit is $\ph(\xi)$, proving the theorem.
	\end{proof}
	
	We now begin the proof of Theorem~\ref{thm:continuity}. The first point is to
	transfer the obstacle inequality to one deformed domain while keeping every
	comparison on the fixed closure.
	\begin{lemma}\label{lem:deformedbound}
	Under the hypotheses of Theorem~\ref{thm:continuity}, put
	\[
	D_t\coloneqq\Phi_t(\Om),
	\qquad
	u_t\coloneqq G_{\ph,\Om}\circ\Phi_t^{-1}
	\quad\text{on }D_t.
	\]
	Then:
	\begin{enumerate}[label=\textup{(\roman*)}]
		\item for every $\eps>0$ and all sufficiently small $t>0$,
		\begin{equation}\label{eq:deformedbound}
			u_t\le\ph+\eps
			\qquad\text{on }D_t\cap\clOm;
		\end{equation}
		\item for every $z\in\Om$,
		\begin{equation}\label{eq:recovery}
			G_{\ph,\Om}(z)=\limsup_{t\downarrow0}u_t(z).
		\end{equation}
	\end{enumerate}
\end{lemma}
	\begin{proof}
		Suppose \eqref{eq:deformedbound} fails. There are $t_j\downarrow0$ and
		$y_j\in D_{t_j}\cap\clOm$ with
		\begin{equation}\label{eq:contradictbound}
			G_{\ph,\Om}(x_j)>\ph(y_j)+\eps,
			\qquad x_j\coloneqq\Phi_{t_j}^{-1}(y_j)\in\Om.
		\end{equation}
		The right side is finite. After taking a subsequence, compactness and
		uniform convergence of the maps give $x_j\to a\in\clOm$ and $y_j\to a$.
		If $a\notin A$, continuity at the finite value $\ph(a)$ implies
		$\ph(x_j)-\ph(y_j)\to0$. This contradicts
		$G_{\ph,\Om}(x_j)\le\ph(x_j)$ and \eqref{eq:contradictbound}.
		If $a\in A\cap\Om$, local boundedness of $G_{\ph,\Om}$ bounds the left
		side, whereas $\ph(y_j)\to+\infty$. This is again impossible.
		Finally, if $a\in A\cap\partial\Om$, the points in
		\eqref{eq:contradictbound} satisfy precisely (D1), including the possibility
		that some $x_j$ belong to $A$. Equation \eqref{eq:compatibility} contradicts
		\eqref{eq:contradictbound}. This proves \eqref{eq:deformedbound}.
		
		For \eqref{eq:recovery}, Lemma~\ref{lem:compactcuts}(i) first ensures
		that the fixed point $z$ belongs to $D_t$ for all sufficiently small $t$,
		so $u_t(z)$ is well defined. The inverse trajectory
		$t\mapsto\Phi_t^{-1}(z)$ then lies in $\Om$ and has a holomorphic
		complexification near zero. Since $z$ is an interior point and the
		complexified trajectory is continuous at $0$, after shrinking the parameter
		disc its image is still contained in $\Om$. Applying
		Lemma~\ref{lem:regularization}(ii) to $G_{\ph,\Om}$ gives
		\eqref{eq:recovery}.
	\end{proof}
	
	The preceding estimate becomes useful after cutting away the part of the moving
	boundary that meets $\clOm$. Combining this cut with the local barriers yields
	the continuous minorants that drive the rest of the proof.
	\begin{lemma}\label{lem:gluing}
	Under the hypotheses of Theorem~\ref{thm:continuity}, the following assertions
	hold.
	\begin{enumerate}[label=\textup{(\roman*)}]
		\item For every $z_0\in\Om$ and every $\eps>0$, there exists
		$v\in\cclass$ such that
		\begin{equation}\label{eq:pointminorant}
			v(z_0)>G_{\ph,\Om}(z_0)-4\eps.
		\end{equation}
		\item If all barriers at $S_\Phi$ are neighbourhood barriers, then the
		function $v$ in \textup{(i)} may be chosen in $\nclass$.
	\end{enumerate}
\end{lemma}
	\begin{proof}
		We first choose the barrier and fix one sufficiently small deformation.
		Apply Lemma~\ref{lem:barrierminorants} with $S=S_\Phi$. Choose
		$u_S\in\cclass$ and a relative open neighbourhood $W$ of $S_\Phi$ in
		$\clOm$ such that
		\begin{equation}\label{eq:bnearS}
			u_S>\ph-\eps\text{ on }W,\qquad
			\overline W\cap(A\cup\{z_0\})=\varnothing.
		\end{equation}
		Here and in this proof, the closure of $W$ is taken in $\clOm$.
		If $S_\Phi=\varnothing$, take $u_S=m\le\inf_{\clOm}\ph$ and
		$W=\varnothing$.
		
		Lemma~\ref{lem:compactcuts}(ii) gives, for all sufficiently small $t$,
		\begin{equation}\label{eq:boundaryinW}
			\partial D_t\cap\clOm\subset W.
		\end{equation}
		Using Lemma~\ref{lem:deformedbound}, choose a single such $t$ for which
		$z_0\in D_t$, $u_t\le\ph+\eps$ on $D_t\cap\clOm$, and
		$u_t(z_0)>G_{\ph,\Om}(z_0)-\eps$.
		
		Next we cut away the part of the moving boundary and regularize on the
		remaining compact piece.
		By Lemma~\ref{lem:compactcuts}(iii),
		\[
		L\coloneqq(\clOm\setminus W)\cap D_t
		\]
		is compactly contained in $D_t$ and contains $z_0$. Choose an open set
		$D$ with $L\subset D\Subset D_t$. Then
		\begin{equation}\label{eq:cutboundary}
			\Gamma\coloneqq\partial D\cap\clOm\subset W.
		\end{equation}
		No connectedness, smoothness or Runge property of $D$ is required.
		Choose a standard convolution regularization of $u_t$ on a fixed
		neighbourhood of $\overline D$. Applying
		Lemma~\ref{lem:regularization}(i) to control this regularization on the
		compact set $\overline D\cap\clOm$, choose a smooth plurisubharmonic
		function $\widetilde u$ with
		\[
		\widetilde u\ge u_t\text{ near }\overline D,\qquad
		\widetilde u\le\ph+2\eps\text{ on }\overline D\cap\clOm.
		\]
		
		We now glue the regularized function back to the barrier minorant on the
		original domain.
		Define
		\begin{equation}\label{eq:gluedfunction}
			v\coloneqq\begin{cases}
				\max\{u_S,\widetilde u-3\eps\},&\text{on }D\cap\clOm,\\
				u_S,&\text{on }\clOm\setminus D.
			\end{cases}
		\end{equation}
		On the compact gluing set $\Gamma$, equations \eqref{eq:bnearS} and
		\eqref{eq:cutboundary} imply
		\[
		\widetilde u-3\eps\le\ph-\eps<u_S.
		\]
		If $\Gamma\ne\varnothing$, continuity and compactness give
		$\min_\Gamma(u_S-\widetilde u+3\eps)>0$. Hence $v=u_S$ on a relative neighbourhood of
		$\Gamma$ in $\clOm$. If $\Gamma=\varnothing$, there is no gluing point
		on $\clOm$. In either case, \eqref{eq:gluedfunction} is continuous on $\clOm$ and
		plurisubharmonic on $\Om$. On $D\cap\clOm$ one has
		$\widetilde u-3\eps\le\ph-\eps$, while $u_S\le\ph$ on $\clOm$.
		Hence $v\le\ph$ on $\clOm$ and $v\in\cclass$. Moreover,
		\[
		v(z_0)\ge\widetilde u(z_0)-3\eps\ge u_t(z_0)-3\eps
		>G_{\ph,\Om}(z_0)-4\eps.
		\]
		This proves \eqref{eq:pointminorant}.
		
		If the barriers extend to a neighbourhood, the same construction can be
		carried out on a neighbourhood of the closure.
		Suppose the barriers at $S_\Phi$ are neighbourhood barriers. Then $u_S$ is
		continuous and plurisubharmonic on a neighbourhood of $\clOm$.
		The strict inequality on $\Gamma$ persists on an open neighbourhood of
		$\Gamma$ where both $u_S$ and $\widetilde u$ are defined. Since
		$\partial D\setminus O$ is compact and disjoint from $\clOm$ for any
		such neighbourhood $O$ of $\Gamma$, we may shrink a neighbourhood $N_0$
		of $\clOm$ so that $\partial D\cap N_0\subset O$.
		When $\Gamma=\varnothing$, choose $N_0$ disjoint from $\partial D$.
		The same formula, with $\clOm$ replaced by $N_0$, is therefore continuous
		and plurisubharmonic on $N_0$. Its restriction to $\clOm$ is the function
		already constructed, and hence belongs to $\nclass$. This proves the lemma.
	\end{proof}
	
	\begin{proof}[Proof of Theorem~\ref{thm:continuity}]
		Lemma~\ref{lem:gluing} proves \eqref{eq:density} and its neighbourhood
		version. Since the right side of \eqref{eq:density} is a supremum of
		continuous functions, $G_{\ph,\Om}$ is lower semicontinuous.
		Theorem~\ref{thm:basic} already gives upper semicontinuity, hence
		$G_{\ph,\Om}$ is continuous.

		It remains to consider increasing boundary data. Fix $v\in\cclass$ and
		$\delta>0$. The increasing open subsets
		\[
		\{\xi\in\partial\Om:f_j(\xi)>v(\xi)-\delta\}
		\]
		cover the compact boundary, including the points where $\ph=+\infty$.
		For all large $j$, $v-\delta\le f_j$ on the whole boundary, whence
		$\PB_{f_j,\Om}\ge v-\delta$ on $\Om$. Taking the supremum over $v$ and
		letting $\delta\downarrow0$, \eqref{eq:PBbound} gives pointwise convergence.
		To see local uniformity without assuming continuity of each Perron envelope,
		let $L\Subset\Om$. At each point of $L$, choose a continuous minorant within
		$\delta$ of $G_{\ph,\Om}$. Continuity of $G_{\ph,\Om}$ and a finite covering
		of $L$ show that a finite maximum $v\in\cclass$ satisfies
		$v>G_{\ph,\Om}-2\delta$ on $L$. The preceding boundary argument then gives
		\[
		G_{\ph,\Om}-3\delta<\PB_{f_j,\Om}\le G_{\ph,\Om}\quad\text{on }L
		\]
		for all large $j$. This completes the proof.
	\end{proof}
	
	The nonmonotone convergence argument can also be localized by a pluripolar
	perturbation. This gives a direct test at any interior point where a potential
	singular on the bad boundary set remains finite.
	
	\begin{lemma}\label{lem:directtest}
	Assume the continuous-minorant identity \eqref{eq:density} and continuity of
	$G=G_{\ph,\Om}$. Let $f_j\in C(\partial\Om)$ satisfy
	\[
	c\le f_j\le\ph,
	\qquad
	f_j\longrightarrow\ph
	\quad\text{locally uniformly on }\partial\Om\setminus B,
	\qquad
	B=A\cap\partial\Om.
	\]
	Suppose $\psi$ is plurisubharmonic on a connected neighbourhood of $\clOm$,
	$\psi\not\equiv-\infty$, and
	\[
	\psi=-\infty\qquad\text{on }B.
	\]
	If $z_0\in\Om$ and $\psi(z_0)>-\infty$, then for every $\eta>0$ there are
	$u\in\PSH(\Om)\cap C(\clOm)$ and an integer $j_0$ such that
	\[
	u\le f_j\quad\text{on }\partial\Om\quad(j\ge j_0),
	\qquad
	u(z_0)>G(z_0)-\eta.
	\]
\end{lemma}
	\begin{proof}
		Normalize $\psi$ so that smooth decreasing regularizations $\psi_k$ on a
		fixed neighbourhood of $\clOm$ satisfy $\psi_k\le0$ there. Choose
		$v\in\cclass$ with $v(z_0)>G(z_0)-\eta/4$, and choose $\lambda>0$ so small
		that $\lambda\psi(z_0)>-\eta/4$. Since $\psi_k\downarrow-\infty$ on the
		compact set $B$, some fixed $k$ satisfies
		\[
		v+\lambda\psi_k<c-1\quad\text{on }B.
		\]
		The compactness justification is elementary: for any prescribed negative
		level, the open sets where $\psi_k$ is below that level increase and cover
		$B$. A finite subcover gives one value of $k$ valid on all of $B$.
		
		Set $u\coloneqq v+\lambda\psi_k-\eta/4$. By continuity, $u<c$ on a relative open
		neighbourhood $O$ of $B$ in $\partial\Om$. On the compact complement,
		$u\le v-\eta/4\le\ph-\eta/4$. Local uniform convergence of the boundary
		data therefore gives $u\le f_j$ on $\partial\Om\setminus O$ for all large
		$j$. On $O$ the same inequality follows from $f_j\ge c$. At the interior
		point, $\psi_k(z_0)\ge\psi(z_0)$ gives
		\[
		u(z_0)>G(z_0)-3\eta/4.
		\]
		This proves the assertion.
	\end{proof}
	
	The preceding perturbation naturally singles out the interior points that cannot
	be separated from the singular boundary set by a neighbourhood
	plurisubharmonic potential. We record this set explicitly.
	Let $\mathcal P_B$ consist of the functions $\psi$
	plurisubharmonic on a connected neighbourhood of $\clOm$, not identically
	$-\infty$, and satisfying $\psi|_B=-\infty$. Assume this class is nonempty,
	and put
	\begin{equation}\label{eq:polarhull}
		\mathcal H_{\Om}(B)\coloneqq
		\{z\in\Om:\psi(z)=-\infty\text{ for every }\psi\in\mathcal P_B\}.
	\end{equation}
	This hull uses potentials on neighbourhoods of $\clOm$; it is not being
	identified with the domain-relative pluripolar hull in \cite{DLH2026}.
	
	\begin{corollary}\label{cor:polarhull}
		Assume \eqref{eq:standing}, the continuous-minorant identity
		\eqref{eq:density}, and continuity of $G=G_{\ph,\Om}$.
		Let $B\coloneqq A\cap\partial\Om$, assume $\mathcal P_B\ne\varnothing$, and let
		$f_j\in C(\partial\Om)$ satisfy $c\le f_j\le\ph$ for a fixed real number
		$c$, with $f_j\to\ph$ locally uniformly on $\partial\Om\setminus B$.
		Then $\PB_{f_j,\Om}\to G$ uniformly on every compact subset of
		$\Om\setminus\mathcal H_{\Om}(B)$.
	\end{corollary}
	\begin{proof}
		At each point $z_0$ of such a compact set, choose a potential in
		$\mathcal P_B$ which is finite at $z_0$ and apply
		Lemma~\ref{lem:directtest}. By continuity of its test function and of $G$,
		the lower estimate persists in a neighbourhood of $z_0$. Finitely many of
		these neighbourhoods cover the compact set, and the maximum of their
		finitely many index thresholds works for every point. The upper estimate
		is \eqref{eq:PBbound}, which completes the proof.
	\end{proof}
	
	\begin{corollary}\label{cor:completepolar}
		Assume the hypotheses of Theorem~\ref{thm:continuity}. Suppose
		$B=A\cap\partial\Om$ is contained in a complete pluripolar set
		$E=\{\psi=-\infty\}$, where $\psi\in\PSH(U)$, $\psi\not\equiv-\infty$, on a
		connected neighbourhood $U$ of $\clOm$. If $f_j\in C(\partial\Om)$ have a
		common finite lower bound, satisfy $f_j\le\ph$, and converge locally
		uniformly to $\ph$ on $\partial\Om\setminus A$, then
		\[
		\PB_{f_j,\Om}\longrightarrow G_{\ph,\Om}
		\quad\text{uniformly on compact subsets of }\Om\setminus E.
		\]
	\end{corollary}
	\begin{proof}
		The given potential $\psi$ belongs to $\mathcal P_B$, so
		$\mathcal H_{\Om}(B)\subset E$. The assertion follows immediately from
		Corollary~\ref{cor:polarhull}.
	\end{proof}
	
	There is a complementary way to avoid an exceptional set altogether: one can
	require convergence at every finite truncation level of the boundary data. The
	next criterion is purely topological on the boundary and is useful for oscillating
	sequences.
	
	\begin{proposition}\label{prop:truncated}
		Let $X$ be a compact metric space, let $\varphi:X\to(-\infty,+\infty]$ be
		continuous, and let $f_j\in C(X)$ satisfy $c\le f_j\le\varphi$. The following
		conditions are equivalent.
		\begin{enumerate}[label=\textup{(\alph*)}]
			\item For every $M\in\R$, the functions $\min\{f_j,M\}$ converge uniformly
			on $X$ to $\min\{\varphi,M\}$.
			\item For every finite $v\in C(X)$ with $v\le\varphi$ and every $\delta>0$,
			one has $f_j\ge v-\delta$ on $X$ for all sufficiently large $j$.
			\item If $j_k\to\infty$ and $x_k\to x\in X$, then
			$\liminf_k f_{j_k}(x_k)\ge\varphi(x)$.
		\end{enumerate}
		If $X=\partial\Om$, $\varphi=\ph|_{\partial\Om}$, and \eqref{eq:density} holds
		with $G$ continuous, these conditions imply
		\[
		\PB_{f_j,\Om}\longrightarrow G\quad\text{locally uniformly on }\Om.
		\]
	\end{proposition}
	\begin{proof}
		Assume (a). Given $v$ as in (b), choose $M>\max_X v$. Then
		$\min\{\varphi,M\}\ge v$, so uniform convergence gives
		$f_j\ge\min\{f_j,M\}\ge v-\delta$ for large $j$. Thus (b) holds.
		Apply (b) to $v=\min\{\varphi,M\}$ and let $x_k\to x$ along any indicated
		sequence. Continuity gives
		$\liminf_k f_{j_k}(x_k)\ge\min\{\varphi(x),M\}-\delta$.
		First let $\delta\downarrow0$, then let $M\to\infty$. This proves (c).
		
		Finally, if (a) fails for a fixed $M$, then $f_j\le\varphi$ and compactness
		give, after passing to a subsequence, $x_j\to x$ and $\eps>0$ with
		\[
		\min\{f_j(x_j),M\}
		<\min\{\varphi(x_j),M\}-\eps.
		\]
		Condition (c) implies that the liminf of the left side is at least
		$\min\{\varphi(x),M\}$, while continuity makes the right side tend to
		$\min\{\varphi(x),M\}-\eps$. This contradiction proves (a).
		
		For the envelope assertion, fix $L\Subset\Om$ and $\delta>0$. The
		continuous-minorant identity and a finite covering give $v\in\cclass$
		with $v>G-\delta$ on $L$. Condition (b) applied to its boundary trace
		gives $v-\delta\le f_j$ eventually on the whole boundary. Consequently
		\[
		G-2\delta< v-\delta\le\PB_{f_j,\Om}\le G\quad\text{on }L
		\]
		for all large $j$. Letting $\delta\downarrow0$ completes the proof.
	\end{proof}
	
	The density theorem also has a useful Jensen-measure consequence. We record
	only the fixed-point form of Edwards duality needed here. For $z\in\Om$, put
	\[
	J_z^{c,\partial}\coloneqq\left\{\mu:\begin{array}{l}
	\mu\text{ is a probability measure on }\partial\Om,\\[2pt]
	v(z)\le\displaystyle\int_{\partial\Om}v\,d\mu
	\quad\text{for every }v\in\PSH(\Om)\cap C(\clOm)
	\end{array}\right\}.
	\]
	\begin{lemma}\label{lem:Edwards}
		For every $z\in\Om$ the class $J_z^{c,\partial}$ is nonempty and compact,
		and for every $f\in C(\partial\Om)$,
		\begin{equation}\label{eq:Edwards}
		\Pc_{f,\Om}(z)
		\coloneqq\sup\{v(z):v\in\PSH(\Om)\cap C(\clOm),\ v\le f\text{ on }\partial\Om\}
		=\min_{\mu\in J_z^{c,\partial}}\int_{\partial\Om}f\,d\mu.
		\end{equation}
	\end{lemma}
\begin{proof}
		We first verify that boundary-supported Jensen measures exist, without
		assuming regularity of the boundary. Let $\omega_z$ be the harmonic measure
		at $z$ for the bounded domain $\Om$, viewed as a probability measure on
		$\partial\Om$. The submean (or stopped Brownian-motion) inequality for a
		continuous subharmonic function on $\clOm$ gives
		$v(z)\le\int_{\partial\Om}v\,d\omega_z$ for every
		$v\in\PSH(\Om)\cap C(\clOm)$. Thus
		$\omega_z\in J_z^{c,\partial}$, and its integral of $f$ is finite.
		
		Extend $f$ to the lower semicontinuous, bounded-below function
		$\widetilde f$ on $\clOm$ which equals $f$ on $\partial\Om$ and $+\infty$
		on $\Om$. Edwards' theorem for lower semicontinuous extended-real data,
		applied to the cone $\PSH(\Om)\cap C(\clOm)$, gives the infimum over its
		full Jensen class on $\clOm$; see \cite{Edwards} and compare
		\cite[Proposition~2.1]{DW}. The boundary measure $\omega_z$ shows that
		this infimum is finite. The full Jensen class is weakly compact, and the
		integral of $\widetilde f$ is lower semicontinuous, so the infimum is
		attained. Every minimizer has finite integral and hence gives zero mass to
		$\Om$. Conversely, a boundary-supported Jensen measure for the cone is
		exactly a member of $J_z^{c,\partial}$. This proves \eqref{eq:Edwards}.
		Finally, $J_z^{c,\partial}$ is weakly closed in the space of probability
		measures on the compact set $\partial\Om$, because every test function has a
		continuous boundary trace. It is therefore compact.
	\end{proof}
	
	\begin{proposition}\label{prop:unboundeddual}
		Under the hypotheses of Theorem~\ref{thm:continuity}, the boundary-supported Jensen class of
		Lemma~\ref{lem:Edwards} satisfies
		\begin{equation}\label{eq:unboundeddual}
			G_{\ph,\Om}(z)=
			\min_{\mu\in J_z^{c,\partial}}\int_{\partial\Om}\ph\,d\mu,
			\qquad z\in\Om.
		\end{equation}
		Every minimizing measure has zero mass on $B=A\cap\partial\Om$ and finite
		integral in \eqref{eq:unboundeddual}.
	\end{proposition}
	\begin{proof}
		Let $\ph_N\coloneqq\min\{\ph,N\}$ on the boundary. The same compact-cover argument
		as in the monotone approximation argument of Theorem~\ref{thm:continuity}, applied to the envelope $\Pc$
		over continuous competitors, gives $\Pc_{\ph_N,\Om}(z)\uparrow G(z)$. Choose minimizing measures
		$\mu_N\in J_z^{c,\partial}$ for $\ph_N$ and a weakly convergent subsequence
		with limit $\mu$ in that compact class. For fixed $M$ and all $N\ge M$,
		\[
		\int \ph_M\,d\mu_N\le\int \ph_N\,d\mu_N=\Pc_{\ph_N,\Om}(z).
		\]
		Pass to the subsequential limit, then let $M\to\infty$. Monotone
		convergence after subtracting a lower bound yields
		$\int\ph\,d\mu\le G(z)$. Conversely, for every $\nu\in J_z^{c,\partial}$
		and $v\in\cclass$, one has $v(z)\le\int v\,d\nu\le\int\ph\,d\nu$.
		Taking the supremum over $v$ gives the reverse bound. Thus the minimum is
		attained and equals the finite value $G(z)$. Finally, if
		$\nu\in J_z^{c,\partial}$ is any minimizing measure, then
		$\int_{\partial\Om}\ph\,d\nu=G(z)<+\infty$. Since $\ph=+\infty$
		on $B$, necessarily $\nu(B)=0$. This proves the assertion for every
		minimizing measure, not only for the subsequential limit constructed above.
	\end{proof}
	
	\section{Hull representation, stability and uniqueness}\label{sec:hullproofs}
	
	\begin{proof}[Proof of Theorem~\ref{thm:hull}]
		Take $f_j=\min\{\ph,j\}$ on $\partial\Om$ and let $L_j$ be its boundary
		disk bundle as in Lemma~\ref{lem:Cauchyhull}. The compact sets $L_j$ decrease
		to $K_\partial$. By the monotone approximation conclusion of Theorem~\ref{thm:continuity},
		$\PB_{f_j,\Om}\uparrow G_{\ph,\Om}$. Lemmas~\ref{lem:Cauchyhull}(i)
		and~\ref{lem:hullcompactness} imply
		\begin{equation}\label{eq:tubeinhull}
			\{(z,w):z\in\Om,\ |w|\le e^{-G_{\ph,\Om}(z)}\}
			\subset\hull{K_\partial}{U\times\C}\subset\hull K{U\times\C}.
		\end{equation}
		
		For the reverse inclusion, projection to $U$ shows that
		$\hull K{U\times\C}\subset\clOm\times\C$, because $\clOm$ is
		$\OO(U)$-convex. Let $v\in\nclass$ and choose a continuous
		plurisubharmonic extension to a neighbourhood of $\clOm$, still denoted by
		$v$. Then $u_v(z,w)\coloneqq\log|w|+v(z)$ is plurisubharmonic on a neighbourhood of this hull
		and is at most zero on $K$. This remains true when $w=0$, where $u_v=-\infty$.
		Lemma~\ref{lem:localhull} gives $u_v\le0$ on the hull. Taking the supremum
		over $v\in\nclass$ and using Theorem~\ref{thm:continuity} gives the reverse
		inclusion in \eqref{eq:hullidentity}.
		
		Finally, $G_{\ph,\Om}\le\ph$, so every point of $K$ with base in $\Om$
		belongs to the left set in \eqref{eq:tubeinhull}. The remaining points of
		$K$ are in $K_\partial$. Hence $K\subset\hull{K_\partial}{U\times\C}$,
		which proves \eqref{eq:fullhullequality}.
	\end{proof}
	
	\begin{proof}[Proof of Corollary~\ref{cor:minimal}]
		The superharmonic minimum principle gives $\ph\ge m$ on $\Om$, so
		$G_{\ph,\Om}\ge m$. Put $R\coloneqq e^{-m}$. The set
		$T\times\{R\}$ is contained in $K_\partial$. If $z\in\hull T U$, then
		$(z,R)\in\hull{K_\partial}{U\times\C}$: apply the defining hull inequality
		to the holomorphic function $\zeta\mapsto H(\zeta,R)$ for each
		$H\in\OO(U\times\C)$. For $z\in\Om$, Theorem~\ref{thm:hull} therefore gives
		$R\le e^{-G_{\ph,\Om}(z)}$, or $G_{\ph,\Om}(z)\le m$.
		Together with the opposite inequality this proves one inclusion in the
		asserted identity. For the converse, we localize the top horizontal slice.
		
		On the disk $|w|\le R$, the affine polynomial
		$f_R(w)\coloneqq(1+w/R)/2$ satisfies $|f_R|\le1$, and equality holds only at $w=R$.
		If $(z,R)\in\hull{K_\partial}{U\times\C}$ and $h\in\OO(U)$, then for
		every positive integer $N$ the hull inequality for $h(\zeta)f_R(w)^N$ gives
		\begin{equation}\label{eq:peaklocalization}
			|h(z)|\le\max_{(\zeta,w)\in K_\partial}|h(\zeta)|\,|f_R(w)|^N.
		\end{equation}
		The maxima on the right decrease to $\max_T|h|$. Indeed, outside any
		relative neighbourhood of $T\times\{R\}$ in $K_\partial$, the continuous
		function $|f_R|$ is bounded by some $\theta<1$. Its contribution tends to
		zero as $N\to\infty$. Shrinking the neighbourhood gives the claimed
		limit, whereas points of $T\times\{R\}$ give the reverse inequality.
		Consequently $|h(z)|\le\max_T|h|$ for every $h\in\OO(U)$, and
		$z\in\hull T U$.
		
		The reverse inclusion in the slice follows by testing $H(\cdot,R)$, as in
		the first part of the proof. We have therefore shown
		\begin{equation}\label{eq:topslice}
			\hull{K_\partial}{U\times\C}\cap\{w=R\}=
			\{(z,R):z\in\hull T U\}.
		\end{equation}
		If $G(z)=m$, Theorem~\ref{thm:hull} places $(z,R)$ in that hull, so
		\eqref{eq:topslice} gives $z\in\hull T U$. The proof of the full equality
		is complete.
	\end{proof}
	
	The peak-fibre argument above is the one used in
	\cite[Proposition 3.2]{DiGreen}.
	
	The outer convergence assumption in Corollary~\ref{cor:stability} is essential,
	even for balls and a fixed logarithmic obstacle, as the following example shows.
	\begin{example}\label{ex:stability}
		Let $B_R\subset\C^n$ be the ball of radius $R$, write $z\coloneqq(z_1,z')$, and put
		$\ph(z)\coloneqq C-\log|z_1|$. Then
		\begin{equation}\label{eq:ballformula}
			G_{\ph,B_R}(z)=C-\frac12\log(R^2-\|z'\|^2).
		\end{equation}
		Indeed the right side is plurisubharmonic and at most $\ph$ on the ball.
		For the reverse inequality, fix $z'$ and restrict any admissible minorant
		to the disk $|z_1|<\sqrt{R^2-\|z'\|^2}$. The one-variable maximum principle
		on concentric smaller circles, followed by passage to their limiting radius,
		bounds that minorant by the right side. This proves \eqref{eq:ballformula}.
		
		Take $U=B_3$, $V=B_{5/2}$, $\Om=B_1$, $\Om_j=B_2$ for every $j$, and
		$\ph\coloneqq\log4-\log|z_1|$ and $\ph_j\coloneqq\ph$. Then
		$\clOm\subset V\Subset U$ and $\overline{\Om_j}\subset V$. These obstacles are
		positive, continuous in the extended sense, and superharmonic on $U$.
		Exponential convergence on $\overline V$ is exact, and every compact subset
		of $\Om$ lies in every $\Om_j$.
		Nevertheless,
		\[
		G_{\ph,\Om}(0)=\log4,\qquad G_{\ph_j,\Om_j}(0)=\log2.
		\]
		Thus inner domain control and equality of the obstacles do not imply
		convergence, even for balls.
	\end{example}
	
	\begin{proof}[Proof of Corollary~\ref{cor:stability}]
		Write $G_j\coloneqq G_{\ph_j,\Om_j}$ and $G\coloneqq G_{\ph,\Om}$. We first establish local
		uniform upper bounds for $G_j$. Fix $p\in\Om$ and choose $r>0$ with
		$\overline{B(p,4r)}\subset\Om$. Since $\ph$ is locally integrable, there is
		$q\in B(p,r)$ with $\ph(q)<\infty$. By the inner domain condition,
		$q\in\Om_j$ for all sufficiently large $j$, and the exponential convergence
		then gives $e^{-\ph_j(q)}\to e^{-\ph(q)}>0$. Thus, after discarding finitely
		many terms, $\ph_j|_{\Om_j}\not\equiv+\infty$ and Theorem~\ref{thm:basic}
		applies to every $G_j$. In particular $\ph_j(q)$ is bounded. The exponential
		convergence also gives a
		common finite lower bound $c_0$ for $\ph_j$ on $\overline V$ for all
		sufficiently large $j$, because the continuous functions $e^{-\ph_j}$ are then
		uniformly bounded above there. Hence the obstacle envelopes are uniformly
		bounded from below. For all large $j$, the ball $B(q,3r)$ and its closure are
		contained in $\Om_j$. The supermean inequality gives
		\[
		\int_{B(q,3r)}\ph_j\,dV
		\le |B(q,3r)|\ph_j(q)\le C_p.
		\]
		For $x\in B(p,r)$, $B(x,r)\subset B(q,3r)$ and $G_j\le\ph_j$. Since
		$\ph_j\ge c_0$ on $B(q,3r)$,
		\[
		\int_{B(x,r)}\ph_j\,dV
		\le C_p-c_0\bigl(|B(q,3r)|-|B(x,r)|\bigr).
		\]
		The submean inequality for $G_j$ therefore gives a uniform upper bound on
		$B(p,r)$. This proves the required local bounds, including near $A$.
		
		It follows that $u\coloneqq(\limsup_j G_j)^*$ is plurisubharmonic on $\Om$.
		On each compact subset where $\ph$ is finite, the convergence of the
		exponentials implies uniform convergence $\ph_j\to\ph$. Thus $u\le\ph$
		there, while on $A$ the inequality is automatic. Hence $u\le G$.
		Hartogs' lemma, using the continuity of $G$, now gives for every $L\Subset\Om$
		and every $\eps>0$
		\begin{equation}\label{eq:stabilityupper}
			G_j\le G+\eps\quad\text{on }L\quad\text{for all sufficiently large }j.
		\end{equation}
		One may obtain the stated form of Hartogs' lemma by covering $L$ with
		finitely many small balls on which the oscillation of $G$ is less than
		$\eps/3$ and applying the constant upper-bound version on each ball.
		
		For the lower estimate, use the neighbourhood-minorant conclusion of
		Theorem~\ref{thm:continuity}. Fix $v\in\nclass$ and $\delta>0$, and let
		$W$ be a neighbourhood of $\clOm$ on which $v$ is defined continuously
		and plurisubharmonically. Choose an open neighbourhood $W_0$ of $\clOm$
		with $\overline{W_0}\Subset W\cap V$. Assumption (c), together with
		$\overline{\Om_j}\subset\overline V$ and compactness of $\overline V$, implies
		$\overline{\Om_j}\subset W_0$ for all sufficiently large $j$: otherwise
		we could choose $y_j\in\overline{\Om_j}\setminus W_0$ and obtain, after
		passing to a subsequence, a cluster point in $\clOm\setminus W_0$, contrary
		to $\clOm\subset W_0$. In particular, $v$ is uniformly bounded on the
		relevant domains. We claim that
		\begin{equation}\label{eq:stableminorant}
			v-\delta\le\ph_j\quad\text{on }\overline{\Om_j}
			\quad\text{for all sufficiently large }j.
		\end{equation}
		Otherwise choose $y_j\in\overline{\Om_j}$ with
		$\ph_j(y_j)<v(y_j)-\delta$. A subsequence converges to $y\in\clOm$.
		The values on the right are bounded, so $e^{-\ph_j(y_j)}$ stay bounded
		away from zero. Uniform convergence of the radii and their continuity give
		$e^{-\ph_j(y_j)}\to e^{-\ph(y)}>0$. Hence $\ph(y)$ is finite and
		$\ph_j(y_j)\to\ph(y)$. Passing to the limit gives
		$\ph(y)\le v(y)-\delta$, contrary to $v\le\ph$ on $\clOm$.
		This proves \eqref{eq:stableminorant}.
		
		Since $v-\delta$ is an admissible minorant on $\Om_j$, we obtain
		$G_j\ge v-\delta$ there. Given a compact $L\Subset\Om$, finitely many
		neighbourhood minorants have a maximum $v$ with $v>G-\delta$ on $L$.
		Thus $G_j>G-2\delta$ on $L$ eventually. Together with
		\eqref{eq:stabilityupper}, this proves the asserted convergence.
	\end{proof}
	
	\begin{remark}
		Under the additional hypotheses of Theorem~\ref{thm:hull}, the lower
		stability estimate also follows from the hull method of \cite{DiGreen}.
		Set
		\[
		\begin{aligned}
		K_j&\coloneqq\{(z,w):z\in\overline{\Om_j},\ |w|\le e^{-\ph_j(z)}\},\\
		K&\coloneqq\{(z,w):z\in\clOm,\ |w|\le e^{-\ph(z)}\}.
		\end{aligned}
		\]
		These sets lie in one fixed compact subset of $U\times\C$, and every
		cluster point of $K_j$ belongs to $K$.
		Lemma~\ref{lem:Cauchyhull}(ii) gives
		$(z,e^{-G_j(z)})\in\hull{K_j}{U\times\C}$ for $z\in\Om_j$.
		If $z_j\to z\in\Om$ and $G_j(z_j)\le G(z)-\eps$, a subsequential limit
		$(z,w)$ satisfies $w\ge e^{-G(z)+\eps}$.
		Lemma~\ref{lem:hullcompactness} puts $(z,w)$ in $\hull K{U\times\C}$,
		contrary to \eqref{eq:hullidentity}.
	\end{remark}
	
\begin{proof}[Proof of Theorem~\ref{thm:uniqueness}]
		Theorems~\ref{thm:basic} and~\ref{thm:continuity} show that $G=G_{\ph,\Om}$
		belongs to the stated class. Conversely, let $u$ belong to it. Since $u\le\ph$,
		the definition of $G$ gives $u\le G$. In particular $u$ is locally bounded
		above, while its assumed lower bound makes it locally bounded on both sides.
		Thus its Bedford--Taylor Monge--Amp\`ere measure is defined, and the
		assumption $(\ddc u)^n=0$ implies that $u$ is maximal.
		
		Choose $\psi\in\PSH(U)$, $\psi\not\equiv-\infty$, with
		$A\cap\partial\Om\subset E\coloneqq\{\psi=-\infty\}$. After shrinking $U$ around
		$\clOm$ and subtracting a constant, assume $\psi\le0$ there.
		For $v\in\cclass$ and $\eps>0$ consider $v+\eps\psi$ on $\Om$.
		At $\xi\in\partial\Om\setminus A$, the boundary trace of $u$ and $v\le\ph$
		give
		\[
		\limsup_{\Om\ni z\to\xi}(v(z)+\eps\psi(z)-u(z))\le0.
		\]
		At $\xi\in A\cap\partial\Om$, upper semicontinuity at the value $-\infty$
		gives $\psi(z)\to-\infty$. Since $v$ is bounded and $u$ is bounded from below,
		the same limsup is $-\infty$. The function $u$ is locally bounded and maximal,
		so Lemma~\ref{lem:comparison} yields $v+\eps\psi\le u$ on $\Om$.
		Let $\eps\downarrow0$ on $\Om\setminus E$, then take the supremum over $v$.
		Equation \eqref{eq:density} gives $G\le u$ there. Thus $u=G$ on
		$\Om\setminus E$. Since $\psi\not\equiv-\infty$ on the connected
		neighbourhood $U$, its $-\infty$ locus $E$ has Lebesgue measure zero.
		Consequently $u=G$ almost everywhere on $\Om$. Both functions belong to
		$L^1_{\mathrm{loc}}(\Om)$, so they define the same distribution; a
		plurisubharmonic function is the upper-semicontinuous representative of its
		$L^1_{\mathrm{loc}}$ class (equivalently, it is recovered from its ball means).
		Hence $u=G$ on all of $\Om$.
	\end{proof}
	
	\begin{remark}
		Let $f(e^{i\theta})\coloneqq|1-e^{i\theta}|^{-1/2}$ on $\partial\Delta$ and let
		$u$ be its Poisson integral. This integrable boundary function is continuous
		in the extended sense, and $u$ has boundary values $f$, including the value
		$+\infty$ at $1$. The latter assertion follows by comparison with the Poisson
		integrals of $\min\{f,N\}$ and then letting $N\to\infty$.
		The positive Poisson kernel
		\[
		P(z)\coloneqq\frac{1-|z|^2}{|1-z|^2}
		\]
		has boundary value zero away from $1$, so $u$ and $u+P$ are different harmonic
		functions with the same extended-real boundary trace. The common obstacle
		$u$ excludes $u+P$.
	\end{remark}
	
	\section{Geometric constructions and explicit tests}\label{sec:geometry}
	
	\begin{proof}[Proof of Corollary~\ref{cor:homogeneous}]
		For each $i$, the function
		$\log\sum_j|f_{i,j}|^{\alpha_i}$ is plurisubharmonic, with value $-\infty$
		where all the corresponding linear forms vanish. This follows, for example,
		by applying the convex increasing log-sum-exp function to the
		plurisubharmonic functions $\alpha_i\log|f_{i,j}|$, with regularization at
		zeros. Thus $\Psi_0$ is plurisuperharmonic, continuous in the extended sense,
		and locally integrable. Its singular set is a finite union of proper linear
		subspaces.
		
		A bounded convex domain is regular for the real Dirichlet problem. Let
		$H\in C(\clOm)$ be the harmonic extension of $h$.
		Then $g\coloneqq\Psi_0+H$ satisfies \eqref{eq:standing} on $\Om$.
		For the expansion $\Phi_t(z)\coloneqq(1+t)z$, convexity and $0\in\Om$ imply
		$\clOm\Subset\Phi_t(\Om)$ for $t>0$. Hence $S_\Phi=\varnothing$.
		If $d\coloneqq\sum_i\alpha_i$, homogeneity gives, away from the singular set,
		\[
		g(z)-g((1+t)z)=d\log(1+t)+H(z)-H((1+t)z),
		\]
		for $z,(1+t)z\in\clOm$. Uniform continuity of $H$ on $\clOm$ makes
		this difference tend to zero uniformly on these overlaps. The singular
		linear subspaces are preserved exactly, so the radius criterion also holds
		on their zero fibres. Theorem~\ref{thm:continuity} therefore gives
		continuity of $G_{g,\Om}$.
		Strict convexity supplies a continuous plurisubharmonic barrier at every
		boundary point: the real part of a supporting complex affine functional,
		with its maximum subtracted, peaks only at that point.
		Theorem~\ref{thm:basic} now proves all the asserted boundary limits and the
		homogeneous Monge--Amp\`ere equation.
	\end{proof}
	
	\begin{remark}
		The same argument proves interior continuity of $G_{\Psi_0,\Om}$ on a bounded
		strongly star-shaped domain, meaning $s\clOm\Subset\Om$ for every
		$0\le s<1$.
	\end{remark}
	
	\begin{proof}[Proof of Proposition~\ref{prop:transverse}]
		If $A\cap\partial\Om=\varnothing$, there is nothing to prove. Assume henceforth that this set is nonempty.
		Write $R(x,t)\coloneqq\rho(\Phi_t^{-1}(x))$. Joint $C^1$ regularity gives
		\[
		\partial_t R(x,0)=2\Rea\sum_i\rho_{z_i}(x)
		\left.\partial_t(\Phi_t^{-1})_i(x)\right|_{t=0}.
		\]
		By \eqref{eq:negativeinverse} and compactness of $A\cap\partial\Om$, there
		are $c>0$, $\delta>0$, and a neighbourhood $W$ of that set such that
		$\partial_t R(x,t)\le-c$ for $x\in W$ and $0\le t\le\delta$.
		For $x\in W\cap\clOm$, after shrinking the neighbourhoods if necessary,
		\[
		R(x,t)=\rho(x)+\int_0^t\partial_s R(x,s)\,ds\le-ct<0.
		\]
		Thus $\Phi_t^{-1}(x)\in\Om$, or $x\in\Phi_t(\Om)$.
		In particular $W\cap\clOm\cap\Phi_t(\partial\Om)=\varnothing$ for such $t$,
		which proves $S_\Phi\cap A=\varnothing$ and the stronger inclusion.
	\end{proof}
	
	\begin{proof}[Proof of Lemma~\ref{lem:divisor}]
		For each fixed $z\in\{f=0\}$, the holomorphic function
		$t\mapsto f(\Phi_t(z))$ vanishes for small real $t\ge0$, hence for small
		complex $t$ by the identity theorem. Since $f$ is reduced, the ideal sheaf
		of its zero set is generated by $f$. Division in the product of the
		$z$-neighbourhood and the parameter disk gives a holomorphic quotient $\theta(z,t)$; the quotients
		agree on overlaps since they agree off $\{f=0\}$. Equivalently, this is the
		principal reduced-divisor case of holomorphic division; see \cite{Ch}.
		At $t=0$, $f=\theta(\cdot,0)f$, so $\theta(\cdot,0)=1$. Compactness permits shrinking
		the parameter disk until $|\theta(z,t)-1|<1/2$ on a neighbourhood of $\clOm$.
		Writing $\theta_t(z)\coloneqq\theta(z,t)$, we see that $\theta_t$ is a unit and tends uniformly to one. Off the divisor,
		\[
		\log|f(\Phi_t(z))|-\log|f(z)|=\log|\theta_t(z)|,
		\]
		which proves the estimate.
	\end{proof}
	
	\begin{proof}[Proof of Proposition~\ref{prop:graphs}]
		The compact set $\pi(\overline{\Om'})$ is contained in the open set $\pi(U)$.
		Choose a slightly larger relatively compact neighbourhood of this compact
		set, and take complex $t$ small enough that $z'+t\Psi(z')$ stays in $\pi(U)$
		there. Define the holomorphic divided differences
		\begin{equation}\label{eq:divided}
			\theta_j(z',t)\coloneqq\begin{cases}
				\dfrac{f_j(z'+t\Psi(z'))-f_j(z')}{t},&t\ne0,\\[4pt]
				\displaystyle\sum_{i=1}^{n-1}\Psi_i(z')\frac{\partial f_j}{\partial z_i}(z'),&t=0.
			\end{cases}
		\end{equation}
		Holomorphicity at $t=0$ follows either from the Taylor expansion or from
		\[
		\theta_j(z',t)=\int_0^1 df_j(z'+st\Psi(z'))[\Psi(z')]\,ds.
		\]
		The interpolation functions
		\begin{equation}\label{eq:Lagrange}
			L_j(z',z_n)\coloneqq\prod_{\ell\ne j}
			\frac{z_n-f_\ell(z')}{f_j(z')-f_\ell(z')},\qquad
			\Theta(z',z_n,t)\coloneqq\sum_{j=1}^k \theta_j(z',t)L_j(z',z_n)
		\end{equation}
		are holomorphic on the required neighbourhood because the denominators
		never vanish. Since $L_j(z',f_i(z'))=\delta_{ij}$, the formula for $\Phi_t$
		gives the asserted graph identities exactly.
		
		To prove injectivity, choose a fixed bounded neighbourhood $N_1$ of
		$\overline{\Om'}$. On $N_1$ the maps have the form $\Phi_t=\mathrm{id}+tR_t$, where
		$R_t$ and its first derivatives are uniformly bounded on a slightly larger
		neighbourhood. There is $r>0$ such that segments between points of $N_1$
		whose distance is less than $r$ lie in that larger neighbourhood.
		For such pairs $x,y$,
		\[
		|\Phi_t(x)-\Phi_t(y)|\ge(1-|t|M)|x-y|.
		\]
		For pairs with $|x-y|\ge r$, the uniform $C^0$ bound gives instead
		$|\Phi_t(x)-\Phi_t(y)|\ge r-2|t|M_0>0$ for small $t$.
		After making $N_1$ bounded and slightly smaller, these estimates prove
		injectivity on $N_1$. The derivative is also invertible for small $t$.
		The holomorphic inverse function theorem gives a holomorphic inverse and
		holomorphic dependence of the inverse on the parameter locally near $t=0$.
		Uniform convergence to the identity and the required inverse-trajectory
		property follow. Thus these maps form the claimed neighbourhood deformation.
	\end{proof}
	
	\begin{proof}[Proof of Proposition~\ref{prop:graphcontinuity}]
		It suffices to consider $h_j>0$, since replacing $\Psi$ by $-\Psi$ changes
		all the signs. Apply Proposition~\ref{prop:graphs} on a neighbourhood of
		$\clOm$. The reduced product $f=\prod_j(z_n-f_j(z'))$ has disjoint smooth
		components, and the deformation preserves them. Lemma~\ref{lem:divisor}
		therefore gives the stronger uniform version of \eqref{eq:compatibility}
		for $\ph\coloneqq-\log|f|$.
		
		The forward velocity on $A_j$ is
		\[
		\left(\Psi(\xi'),\ \sum_{i=1}^{n-1}\Psi_i(\xi')
		\frac{\partial f_j}{\partial z_i}(\xi')\right).
		\]
		Differentiating $\Phi_t\circ\Phi_t^{-1}=\mathrm{id}$ at $t=0$ shows that
		the inverse velocity is its negative. Consequently
		\[
		\left.\frac{d}{dt}\rho(\Phi_t^{-1}(\xi))\right|_{t=0}=-2h_j(\xi)<0.
		\]
		Proposition~\ref{prop:transverse} yields $S_\Phi\cap A=\varnothing$.
		
		At a boundary point $\xi$, choose $\eta>0$ so small that
		$\rho-\alpha-\eta|z-\xi|^2$ is plurisubharmonic on a neighbourhood of
		$\clOm$. This function is continuous, equals zero at $\xi$, and is strictly
		negative on $\clOm\setminus\{\xi\}$. It is therefore a neighbourhood
		barrier. Theorem~\ref{thm:continuity} applies.
		
		Since $\rho$ is a strictly plurisubharmonic exhaustion, $U$ is Stein. The
		compact sublevel $\{\rho\le\alpha\}$ is therefore $\OO(U)$-convex by
		the plurisubharmonic hull characterization. Since $\alpha$ is a regular value, this compact
		sublevel is a smooth manifold with boundary and therefore has only finitely
		many connected components. Moreover, the closure of each component of
		$\{\rho<\alpha\}$ is exactly the corresponding component of
		$\{\rho\le\alpha\}$: near the regular level set, the function
		$\rho-\alpha$ is a local boundary coordinate, so two distinct interior
		components cannot meet only along $\{\rho=\alpha\}$. Choose pairwise disjoint neighbourhoods of these
		components. A function that is identically zero near one component and
		identically one near the others is holomorphic on their union. By Oka--Weil
		approximation on $\{\rho\le\alpha\}$, it can be approximated within
		$1/3$ by a function in $\OO(U)$. Hence every point of any other component
		is separated from the chosen one by an $\OO(U)$-function. Since the hull of
		the chosen component is already contained in the whole sublevel, it follows
		that $\clOm$ is $\OO(U)$-convex.
		Theorem~\ref{thm:hull} follows. Finally the divisor is complete pluripolar
		on $U$, with defining plurisubharmonic function $\log|f|$, and every finite
		boundary point has a barrier. Theorem~\ref{thm:uniqueness} therefore applies.
	\end{proof}
	
	We now turn to explicit deformations. The first example shows directly that the
	continuity theorem genuinely allows families that are pairwise nonnested.
	\begin{example}\label{ex:nonnested}
		Let $B\coloneqq\{(z_1,z_2):|z_1|^2+|z_2|^2<1\}$ and set $\Om\coloneqq B$. Let
		$\ph\coloneqq-\log|z_1|$, and set
		\begin{equation}\label{eq:nonnestedmap}
			\Phi_t(z_1,z_2)\coloneqq(e^{-t}z_1,e^t z_2),\qquad t\ge0.
		\end{equation}
		The resulting domains are
		\begin{equation}\label{eq:nonnesteddomains}
			D_t\coloneqq\{(y_1,y_2):e^{2t}|y_1|^2+e^{-2t}|y_2|^2<1\}.
		\end{equation}
		For every $t>0$, the first semiaxis has length $e^{-t}<1$ and the second
		has length $e^t>1$. Consequently neither $D_t\subset B$ nor $B\subset D_t$
		holds. More generally, if $0\le s<t$, a point on the first coordinate axis
		can be chosen in $D_s\setminus D_t$, and a point on the second axis can
		be chosen in $D_t\setminus D_s$. Hence the entire family is pairwise
		incomparable under inclusion.
		
		The hyperplane $\{z_1=0\}$ is invariant under $\Phi_t$.
		On the overlap $x\in B$, $\Phi_t(x)\in\overline B$, one has
		\[
		\ph(x)-\ph(\Phi_t(x))=-t\qquad(x_1\ne0).
		\]
		Equivalently, $r(\Phi_t(x))=e^{-t}r(x)$ on this overlap, including its zero
		fibres. Thus (D1) holds.
		To compute the crossing set, write $y\coloneqq\Phi_t(x)$ with $x\in\partial B$.
		Then $y\in\overline B$ precisely when
		\[
		e^{-2t}|x_1|^2+e^{2t}|x_2|^2\le1,
		\quad\text{or equivalently}\quad
		|x_2|^2\le\frac{1}{1+e^{2t}}.
		\]
		It follows that
		\begin{equation}\label{eq:nonnestedcluster}
			S_\Phi=\{\xi\in\partial B:|\xi_2|^2\le1/2\}.
		\end{equation}
		For the reverse inclusion implicit in this equality, points with strict
		inequality can be kept fixed as source points for all small $t$; points
		with equality are approached by source points whose second squared modulus
		is $1/(1+e^{2t})$. Thus every point of the stated set is attained as a
		cluster point. Since $|\xi_1|^2\ge1/2$ on $S_\Phi$, this set is disjoint
		from $A$. Each $\xi\in\partial B$ has the neighbourhood barrier
		$\Re\langle z,\xi\rangle-1$. All hypotheses of Theorem~\ref{thm:continuity}
		are therefore satisfied.
		
		The envelope is
		\begin{equation}\label{eq:nonnestedG}
			G_{\ph,B}(z_1,z_2)=-\frac12\log(1-|z_2|^2).
		\end{equation}
		For fixed $z_2$, every admissible minorant is bounded by the right side
		on the disk $|z_1|<\sqrt{1-|z_2|^2}$, by the one-variable maximum principle.
		The right side is itself a plurisubharmonic minorant, which proves the
		formula. It is finite and continuous across the interior divisor.
		
		For additional comparison, keeping the same obstacle on $D_t$ gives
		\begin{equation}\label{eq:nonnestedGt}
			G_{\ph,D_t}(z_1,z_2)
			=t-\frac12\log(1-e^{-2t}|z_2|^2).
		\end{equation}
		These envelopes do not have one common pointwise order as $t$ increases.
		Their derivative at $t=0$, at a fixed interior point, is
		\[
		\left.\frac{\partial}{\partial t}G_{\ph,D_t}(z)\right|_{t=0}
		=\frac{1-2|z_2|^2}{1-|z_2|^2}.
		\]
		It is positive when $|z_2|^2<1/2$ and negative when $|z_2|^2>1/2$.
		The functions in \eqref{eq:nonnestedGt} converge locally
		uniformly to \eqref{eq:nonnestedG} as $t\to0$.
	\end{example}
	
	The next example concerns the boundary-data side of the theorem. It gives an
	unbounded analogue of the nonconvergence phenomenon from \cite{DieuPB2007}
	and, in the present setting, identifies exactly where convergence may be lost.
	\begin{example}\label{ex:boundarycutoffs}
		In \cite[Proposition~3.1]{DieuPB2007}, uniformly bounded continuous data on
		the unit sphere in $\C^2$ were shown to converge pointwise to zero while
		their Perron--Bremermann envelopes fail to converge to zero. The present example is adapted to an
		unbounded limiting obstacle and makes the resulting exceptional locus explicit.
		Continue with the ball and obstacle of Example~\ref{ex:nonnested}. On
		$\partial B$, put $r\coloneqq|z_1|$ and choose the continuous cutoff
		\[
		\chi(s)\coloneqq
		\begin{cases}
			0,&0\le s\le1,\\
			s-1,&1<s<2,\\
			1,&s\ge2.
		\end{cases}
		\]
		For $j\ge2$, let $\lambda_j\coloneqq j^2$ when $j$ is even and $\lambda_j\coloneqq j$ when
		$j$ is odd, and define
		\begin{equation}\label{eq:cutoffdata}
			f_j\coloneqq\chi(\lambda_j r)\min\{\ph,j\}\quad\text{on }\partial B.
		\end{equation}
		These are finite continuous functions: the truncated obstacle is continuous
		also at $r=0$. They satisfy $0\le f_j\le\ph$, and on each compact subset
		of $\partial B\setminus A$ one has $f_j=\ph$ for all sufficiently large
		$j$. The sequence is not increasing. Indeed, for even $j$ large enough,
		choose a boundary point with $r=1/(j+1)$. Then
		$f_{j+1}=0$, whereas $f_j=\log(j+1)>0$. Nor is it decreasing: at a fixed point with $r=1/4$, one has $f_2=0$,
		whereas $f_j$ eventually equals the positive value $\ph=\log4$.
		
		Every $f_j$ vanishes on the singular boundary circle
		$B_A=\{(0,\zeta):|\zeta|=1\}$. Restrict an admissible Perron competitor
		to the analytic disk $\{(0,\zeta):|\zeta|<1\}$. Its boundary upper limits
		are at most zero. The subharmonic maximum principle, and the admissible
		constant zero for the reverse inequality, give
		\begin{equation}\label{eq:cutofffailure}
			\PB_{f_j,B}(0,z_2)=0\qquad(|z_2|<1).
		\end{equation}
		But \eqref{eq:nonnestedG} is strictly positive there whenever $z_2\ne0$.
		Thus convergence on the whole ball is false for these data. On every
		compact subset of $B\setminus\{z_1=0\}$, it does hold by
		Corollary~\ref{cor:completepolar}.
		
		This example also identifies the intrinsic exceptional set
		\eqref{eq:polarhull}. If a plurisubharmonic function on a neighbourhood of
		$\overline B$ is $-\infty$ on $B_A$, its restriction to the closed axis disk
		is $-\infty$ on the entire disk. To see this, for any $M>0$, upper
		semicontinuity and compactness give a collar of the boundary circle on
		which the restriction is below $-M$; the maximum principle then gives the
		same bound inside. Let $M\to\infty$. Therefore
		$\mathcal H_B(B_A)$ contains the axis disk. The function $\log|z_1|$ gives
		the reverse inclusion. Hence
		\[
		\mathcal H_B(B_A)=\{(0,z_2):|z_2|<1\}.
		\]
	\end{example}
	
	In contrast, $\widetilde f_j=\min\{\ph,j\}-\delta_j$, with
	$\delta_j=(1+(-1)^j/2)/j$, is another nonmonotone sequence from below.
	It satisfies the finite-truncation condition of
	Proposition~\ref{prop:truncated}: for every fixed $M$, once $j>M+1$ one has
	\[
	\bigl\|\min\{\widetilde f_j,M\}-\min\{\ph,M\}\bigr\|_{\partial B}
	\le\delta_j\longrightarrow0.
	\]
	Its envelopes converge locally uniformly on the whole ball, including the
	axis.
	
	The same mechanism applies to invariant divisors beyond hyperplanes. A basic
	family is provided by weighted homogeneous curves.
	\begin{example}\label{ex:cusp}
		In the unit ball $B\subset\C^2$ (with ambient domain $U=\C^2$), let
		$f(z_1,z_2)\coloneqq z_1^m-z_2^n$, where $m,n\ge1$ are integers, and put
		$\ph\coloneqq-\log|f|$. The outward deformation
		\[
		\Phi_t(z_1,z_2)\coloneqq((1+t)^nz_1,(1+t)^mz_2)
		\]
		satisfies $f\circ\Phi_t=(1+t)^{mn}f$ and
		$\ph-\ph\circ\Phi_t=mn\log(1+t)$. Every point of the closed unit ball lies
		strictly inside $\Phi_t(B)$ for $t>0$, so $S_\Phi=\varnothing$.
		Theorems~\ref{thm:continuity}, \ref{thm:hull}, and~\ref{thm:uniqueness}
		apply. This includes singular curves and finite unions of components when
		$m,n$ have a common divisor.
	\end{example}
	
	The graph construction of Proposition~\ref{prop:graphs} also yields examples
	in which the singular set is nonlinear and the boundary cluster set is nonempty.
	\begin{example}\label{ex:curved}
		Let $B\subset\C^2$ be the unit ball, with ambient domain $U=\C^2$. Choose $0<a<1$ and
		$\varepsilon\in\C$ with $|\varepsilon|<1/2$. Set
		\[
		f_\pm(z_1)\coloneqq\varepsilon z_1^2\pm a,\qquad
		F(z_1,z_2)\coloneqq(z_2-\varepsilon z_1^2)^2-a^2,\qquad \ph\coloneqq-\log|F|.
		\]
		The two graphs are disjoint everywhere. The explicit polynomial
		biholomorphisms
		\begin{equation}\label{eq:curveddeformation}
			\Phi_t(z_1,z_2)\coloneqq
			\big((1+t)z_1,\ z_2+\varepsilon((1+t)^2-1)z_1^2\big)
		\end{equation}
		satisfy $F\circ\Phi_t=F$. Their inverse is obtained by first dividing the
		first coordinate by $1+t$, then subtracting the displayed quadratic term.
		Thus the obstacle compatibility is exact.
		
		For $\rho\coloneqq|z_1|^2+|z_2|^2-1$ and forward velocity
		$(z_1,2\varepsilon z_1^2)$, the outward normal expression at a point
		$\xi\in\{F=0\}\cap\partial B$ is
		\begin{align*}
			h(\xi)&=|\xi_1|^2+\Rea(2\varepsilon\xi_1^2\overline{\xi_2})\\
			&\ge(1-2|\varepsilon|)|\xi_1|^2>0.
		\end{align*}
		The last strict inequality holds because $\xi_1=0$ on these graphs would
		force $\xi_2=\pm a$, which lies inside the ball, not on its boundary.
		Proposition~\ref{prop:transverse} proves $S_\Phi\cap\{F=0\}=\varnothing$.
		In fact, if $s\coloneqq(1+t)^2-1>0$ and $y\in\partial B$, then
		\[
		\|\Phi_t(y)\|^2-1
		\ge s(1-2|\varepsilon|)|y_1|^2.
		\]
		Thus the only crossing points in the closed ball are the fixed points
		$y_1=0$, and $S_\Phi=\{(0,w):|w|=1\}$ is genuinely nonempty.
		All boundary points of the ball have neighbourhood barriers, and
		$\overline B$ is polynomially convex. Hence the residual envelope is
		continuous, the disk-bundle hull identities hold, and the obstacle-dominated
		solution is unique.
	\end{example}
	
	\section{Logarithmic obstacles with several generators}\label{sec:scope}
	
	\begin{proposition}\label{prop:tuple}
		Let $F\coloneqq(F_1,\ldots,F_q)$ be holomorphic on a neighbourhood of $\clOm$, and
		assume that $F$ is not identically zero on the connected neighbourhood in question. Put
		\[
		\ph\coloneqq-\log\|F\|+H,
		\]
		where $H\in C(\clOm;\R)$ is superharmonic on $\Om$. Suppose $\Phi_t$ is a
		neighbourhood deformation and that on a fixed neighbourhood of $\clOm$
		there are holomorphic $q\times q$ matrices $M_t$ such that
		\begin{equation}\label{eq:matrixcovariance}
			F(\Phi_t(x))=M_t(x)F(x),\qquad
			\sup_{x\in\clOm}\|M_t(x)-I\|\longrightarrow0.
		\end{equation}
		If $S_\Phi$ avoids $\{F=0\}\cap\partial\Om$ and admits the barriers of
	Theorem~\ref{thm:continuity}, then $G_{\ph,\Om}$ is continuous and the monotone
	approximation conclusion of that theorem holds. The nonmonotone conclusion of
	Corollary~\ref{cor:completepolar} also applies. The same assertion holds
		for a finite positive linear combination of such logarithmic obstacles,
		provided each tuple satisfies \eqref{eq:matrixcovariance} for the same
		family and the cluster set avoids the union of their zero sets.
	\end{proposition}
	\begin{proof}
		The function $\log\|F\|$ is plurisubharmonic, so $\ph$ satisfies
		\eqref{eq:standing}. The singular set is exactly $\{F=0\}\cap\clOm$.
		Write $\eta_t\coloneqq\sup_{\clOm}\|M_t-I\|$. For small $t$, $\eta_t<1/2$ and
		\[
		(1-\eta_t)\|F(x)\|\le\|F(\Phi_t(x))\|
		\le(1+\eta_t)\|F(x)\|.
		\]
		In particular, the common zero set is preserved in both directions wherever
		this identity is used. At nonzero fibres,
		\[
		\left|\log\frac{\|F(\Phi_t(x))\|}{\|F(x)\|}\right|
		\le \max\{-\log(1-\eta_t),\log(1+\eta_t)\}\longrightarrow0.
		\]
		On the overlap $x\in\Om$, $\Phi_t(x)\in\clOm$, uniform continuity of $H$
		gives $H(x)-H(\Phi_t(x))\to0$ uniformly. Thus (D1) follows, with the
		zero-fibre convention included. Moreover,
		$\{F=0\}=\{\log\|F\|=-\infty\}$ is complete pluripolar on the
		connected neighbourhood under consideration. Hence Theorem~\ref{thm:continuity}
		and Corollary~\ref{cor:completepolar} give the stated continuity and
		approximation conclusions. For
		finitely many tuples, sum the estimates with their positive coefficients.
		The union of the corresponding zero sets is again contained in a complete
		pluripolar set, for instance the $-\infty$ locus of the sum of the functions
		$\log\|F^{(\ell)}\|$, and the total error still tends to zero. This proves
		the stated extension.
	\end{proof}
	
	Weighted homogeneous tuples give a simple class of examples. If
	$F_\ell(\Phi_t(x))=e^{d_\ell t}F_\ell(x)$, then the diagonal matrix
	$M_t\coloneqq\operatorname{diag}(e^{d_1t},\ldots,e^{d_qt})$ tends to the identity.
	
	\begin{remark}
		Equation~\eqref{eq:matrixcovariance} is a sufficient hypothesis, not a
		consequence of mere invariance of the common zero set. The distinction is
		already visible for $F(z_1,z_2)\coloneqq(z_1^2,z_2)$ and
		$\Phi_t(z_1,z_2)\coloneqq(z_1,z_2+t z_1)$. The common zero set is $\{0\}$ and is
		preserved. But, for a fixed nonzero $t$, the term $t z_1$ is not in the
		ideal $(z_1^2,z_2)$, so the second transformed component cannot be expressed
		as a holomorphic linear combination of the original components near $0$.
		Also, along $(z_1,0)$, the quotient
		$\|F(\Phi_t(z_1,0))\|/\|F(z_1,0)\|$ is unbounded as $z_1\to0$.
		In the reduced principal-divisor case, Lemma~\ref{lem:divisor} supplies
		the required unit factor.
	\end{remark}
	
	\section*{Acknowledgements}
\noindent 	
	The second named author initiated this work shortly after the passing of Professor
	Nguyen Thanh Van (1943--2015), whose guidance had a lasting influence on his
	mathematical work. The project was completed in August 2026, after a long period
	shaped by both technical difficulties and personal circumstances, and only a few
	months after the passing of his father, Professor Nguyen Van Khue (1940-2026). The authors
	dedicate this paper to the memory of Professors Nguyen Thanh Van and Nguyen Van
	Khue, whose lives and work remain enduring sources of inspiration.
	\begingroup
	\small
	\setlength{\parskip}{0pt}
	
	\endgroup
	
\end{document}